\documentclass[11pt]{amsart}
\usepackage{amscd}
\usepackage{amsfonts}
\usepackage{color}
\usepackage{amsmath,amssymb,amsthm,latexsym}
\usepackage{amscd}

\usepackage{multicol}
\usepackage{enumerate}
\usepackage{graphicx}
\newtheorem{theorem}{Theorem}[section]
\newtheorem{proposition}[theorem]{Proposition}
\newtheorem{definition}[theorem]{Definition}
\newtheorem{corollary}[theorem]{Corollary}
\newtheorem{lemma}[theorem]{Lemma}

\numberwithin{equation}{section}
\theoremstyle{remark}
\newtheorem{remark}[theorem]{Remark}
\newtheorem{example}[theorem]{\bf Example}
\newcommand{\R}{\mathbb{R}}
\newcommand{\D}{\mathbb{D}}

\newcommand{\hh}{\mathbb{H}}
\newcommand{\hhh}{\mathbf{h}}
\newcommand{\SSS}{\mathbb{S}}
\newcommand{\dd}{\mathrm{d}}
\newcommand{\C}{\mathbb{C}}
\begin{document}

\title[Willmore deformations between minimal surfaces in $\hh^{n+2}$  $\&$  $\SSS^{n+2}$]{\bf{Willmore deformations between minimal surfaces in $\hh^{n+2}$ and $\SSS^{n+2}$}}

\author{Changping Wang, Peng Wang}
\thanks{CPW was partly supported by the Project 11831005  of NSFC. PW was partly supported by the Project 11971107 of NSFC. The authors are thankful to Prof. Zhenxiao Xie for value discussions}
\address{College of Mathematics and Informatics, FJKLMAA, Fujian Normal University, Fuzhou 350117, P. R. China}
\email{cpwang@fjnu.edu.cn}
\address{College of Mathematics and Informatics, FJKLMAA, Fujian Normal University, Fuzhou 350117, P. R. China}
\email{pengwang@fjnu.edu.cn,
netwangpeng@hotmail.com}
   
\thanks{This work is partially supported by the Project 11831005 and 11971107 of NSFC. The authors are thankful to Prof. Josef Dorfmeister,  Prof. Shimpei Kobayashi, Prof. Xiang Ma and Prof. Nan Ye for valuable discussions}
\begin{abstract}
	In this paper we show that locally there exists a Willmore deformation between minimal surfaces in $\SSS^{n+2}$ and minimal surfaces in $\hh^{n+2}$, i.e., there exists a smooth family of Willmore surfaces $\{y_t,t\in[0,1]\}$ such that $(y_t)|_{t=0}$ is conformally equivalent to a minimal surface in $\SSS^{n+2}$ and  $(y_t)|_{t=1}$  is conformally equivalent to  a  minimal surface in $\hh^{n+2}$. For some cases the deformations are global. Consider the Willmore deformations of the Veronese two-sphere and its generalizations in $S^4$, for any positive number $W_0\in\R^+$, we construct complete minimal surfaces in $\hh^4$ with Willmore energy being equal to $W_0$.  An example  of complete minimal M\"{o}bius strip in $\hh^4$ with Willmore energy $\frac{6\sqrt{5}\pi}{5}\approx10.733\pi$ is also presented. We also show that all isotropic minimal surfaces in $\SSS^4$ admit Jacobi fields different from Killing fields, i.e., they are not ``isolated".
\end{abstract}\date{\today}
\maketitle



%\tableofcontents

\section{Introduction}

Minimal surfaces in $\hh^n$ are important geometric objects in geometry  \cite{A} and mathematical physics \cite{Mald,DGO,AM2010,AM2015} and attract many attentions from different kind of directions (\cite{C1,C2,Lin}). For instance, in \cite{AM2010} it is shown that the renormalized area introduced by Maldacena in \cite{Mald} can be expressed as the Willmore functional of minimal surfaces in $\hh^n$. Moreover, in \cite{AM2015} Alexakis and Mazzeo discussed in details of the geometry and analysis of complete Willmore surfaces in $\hh^3$ which meet the infinity boundary $\partial_{\infty}\hh^3$ orthogonally. Minimal surfaces in $\hh^n$ can be viewed as special kind of Willmore surfaces, which are the critical surface of the Willmore functional. It is natural to consider them under the framework of Willmore surfaces. In \cite{DoWa10,DoWa12} Dorfmeister and Wang started the study of the global geometry of Willmore surfaces in terms of the harmonic conformal Gauss maps and the DPW method. Such an idea was first introduced by H\'{e}lein in \cite{Helein} ( generalized by Xia-Shen \cite{Xia-Shen}). Moreover, in \cite{Wang-2}, a description of minimal surfaces in space forms as special Willmore surfaces is presented.

In this paper, we continue the study minimal surfaces in $\hh^n$ and $\SSS^n$ along this direction. To begin with, let us first recall the characterization of
minimal surfaces in space forms \cite{Wang-2} briefly. Roughly speaking, the DPW method gives a representation of Willmore surfaces in terms of some Lie-algebra-valued meromorphic 1-form called normalized potential \cite{DPW, Helein, DoWa10, DoWa12}.
Then a Willmore surface being minimal in some space form is equivalent to the Lorenzian orthogonality of some (non-zero) constant real vector $\mathbf v$ with some part of the normalized potential \cite{Wang-2}. The vector $\mathbf v$ being lightlike, timelike or spacelike corresponds to the space form $\R^{n+2}$, $\SSS^{n+2}$ or $\hh^{n+2}$ respectively (See \cite{Wang-2} or Theorem \ref{thm-minimal} of Section 2; Compare also \cite{Helein, Xia-Shen} for a slightly different treatment, where a different harmonic map introduced by \cite{Helein} is used).

A key observation due to  this paper is that the Lorenzian orthogonality is preserved by some complex group action, while the minimality in space forms could be changed. This makes it possible to deform minimal surfaces in $\SSS^{n+2}$ into non-minimal Willmore surfaces and furthermore into minimal surfaces in  $\hh^{n+2}$ or conversely\footnote{It is natural to compare this correspondence with the famous Lawson correspondence \cite{Lawson}. A crucial difference is that from a minimal surface in $\SSS^n$, one can obtain a lot of non-isometric minimal surfaces in $\hh^n$. See Section 5.}:
\begin{theorem} (See Theorem 4.1)  Let $y:U\rightarrow \SSS^{n+2}$ be a minimal surface from a simple connected open subset $U\subset M$. There exists a family of Willmore surfaces $y_t:U'\subset U\rightarrow \SSS^{n+2}$, $t\in[0,\pi]$, such that $y_{t}|_{t=0}=y$ and $y_{t}|_{t=\pi/2}$ is conformally equivalent to a minimal surface in  $\hh^{n+2}$. Here $U'$ is an open subset of $U$.
\end{theorem}
Such a phenomenon is new to the authors' best knowledge. Note that in \cite{BFLPP,BQ}, dressing actions of Willmore surfaces are discussed. But they are different from the actions discussed here since here we use simply elements in the complexified subgroup $K^{\C}$. For a general discussion of dressing actions, we refer to \cite{Gu-Oh, TU, Uh}.\vspace{3mm}

One of the most simple minimal surfaces in $\SSS^{n+2}$ is the Veronese two-sphere in $\SSS^4$. We show explicitly the Willmore deformations for the Veronese two-sphere in $\SSS^4$. Moreover, we obtain a lot of explicit examples of complete minimal disks in $\hh^4$ which are deformed from the Veronese two-sphere and its generalizations:
\begin{example} (of Proposition \ref{prop-h4k})
	Set\begin{equation}
	Y_t=\left(
	\begin{array}{c}
	y_0 \\
	y_1 \\
	y_2 \\
	y_3 \\
	y_4 \\
	y_5 \\
	\end{array}
	\right)
	=\left(
	\begin{array}{ccccc}
	(k-1)(e^{2t}r^{2k+2}+1)+(k+1)(r^{2k}+e^{2t}r^2)\\
	-(k-1)(e^{2t}r^{2k+2}+1)+(k+1)(r^{2k}+e^{2t}r^2)\\
	i e^t\sqrt{k^2-1}(1+r^{2k})(z-\bar z)\\
	e^t\sqrt{k^2-1}(1+r^{2k})(z+\bar z)\\
	i \sqrt{k^2-1}(1-e^{2t}r^{2})(z^{k}-\bar z^{k})\\
	-\sqrt{k^2-1}(1-e^{2t}r^{2})(z^{k}+\bar z^{k})\\
	\end{array}
	\right).
	\end{equation}
	The equation $y_1=0$ gives two circles of $S^2$, which divide $S^2$ into three parts. On each part of them,
	\[y_t=\frac{1}{y_1}\left(
	\begin{array}{ccccc}
	y_0 & y_2 & y_3 & y_4 & y_5 \\
	\end{array}
	\right)\] provides a proper, complete minimal surface in $\hh^4$ with finite Willmore energy.
	Moreover, for  any number $W_0\in\R^+$,  there exist some $k\in\mathbb Z^+\setminus\{1\}$ and $t'\in\R$ such that one of the above three minimal surfaces, has Willmore energy $W_0$. Note that when $k=1$, $y_t$ is in the Willmore deformation family of the Veronese sphere in $S^4$.
\end{example}

\begin{remark}\
	\begin{enumerate}
		\item This is  different from the value distribution of Willmore two-spheres in $S^4$ \cite{Bryant1984,Mon}, where the Willmore energy is always $4\pi k$ for some $k\mathbb Z^+\cup\{0\}$.
		Note that different from the cases discussed in \cite{AM2010,AM2015}, the examples constructed here do not intersect the infinite boundary $\SSS^3_{\infty}=\partial_{\infty}\hh^4$ orthogonally, since there are equivariant and not rotating. But they do intersect the infinite boundary $\SSS^3_{\infty}=\partial_{\infty}\hh^4$  with a constant angle.
		
		\item By embedding $\hh^4$ conformally into $\SSS^4$ via the canonical map (see e.g. \cite{Ba-Bo,BPP,WangCP}) \[x=(x_0,x_1,\cdots,x_4)\mapsto \frac{1}{x_0}(1,x_1,\cdots,x_4),\] the three minimal surfaces form a Willmore immersion from $S^2$ to $\SSS^4$ by crossing the infinite boundary of $\hh^4$, which gives an explicit illustration of Babich and Bobenko's famous construction of Willmore tori (with umbilical circles) in $\SSS^3$ via gluing complete minimal surfaces in $\hh^3$ at the infinite boundary of $\hh^3$ in \cite{Ba-Bo}. A slight difference is that, although here the intersection of these surfaces with the infinite boundary $\SSS^3_{\infty}$ is not orthogonal, the whole surface stays smooth. We refer to Section 5.3 for more details.

	\end{enumerate}
\end{remark}

We also obtain a complete minimal M\"{o}bius strip in $\hh^4$ with Willmore energy $\frac{6\sqrt{5}\pi}{5}\approx10.733\pi$ (see Section 5.5). It can be extended as above to obtain a branched Willmore $\R P^2$ in $S^4$ (Compare \cite{Ishihara}). It is natural to ask the infimum of the Willmore energy of non-oriented complete minimal surfaces in $\hh^n$, in comparison with the famous Willmore conjecture, which is  proved by Marques and Neves \cite{Marques}  for the case of $\SSS^3$. This example shows that the infimum is $\leq\frac{6\sqrt{5}\pi}{5}$.

Using the $K^{\C}$ dressing actions, we can construct concretely a family of isotropic minimal surfaces in $\SSS^4$ for each such surface, which shows that they are not isolated.

\vspace{5mm}
This paper is organized as follows: in Section 2 we will review the basic theory of Willmore surfaces and loop group description of them in terms of their conformal Gauss map. Then in Section 3 we will discuss in details of the $K^{\C}-$dressing of Willmore surfaces in $\SSS^{n+2}$, as well as applications to minimal surfaces in $\SSS^{n+2}$ and $\hh^{n+2}$. Section 4 is a description of  two kind of one parameter group dressing actions on minimal surfaces in $\SSS^{n+2}$ and $\hh^{n+2}$. Then in Section 5 we will focus on examples of complete minimal surfaces in $\hh^4$ with bounded Gauss curvature and finite Willmore energy.
In Section 6 we show that isotropic minimal surfaces in $\SSS^4$ have non-trivial minimal deformations.
The paper is ended by an appendix for the technical proof of a lemma.

\section{Surface theory of Willmore surfaces and the DPW constructions}

In this section we will first recall the basic theory about Willmore surfaces in $\SSS^{n+2}$. Then we will collect the basic DPW theory for harmonic maps in symmetric space and its applications to Willmore surfaces.

\subsection{ Willmore surfaces in $\SSS^{n+2}$}
Here we will follow the treatment for Wilmore surfaces in \cite{BPP,DoWa10,DoWa12,Ma}. {Note that in \cite{Helein,Xia-Shen}, different frames are used in the spirits of \cite{Bryant1984} and \cite{WangCP} respectively. }
Let $ \mathbb{R}^{n+4}_1$ be the Lorentz-Minkowski space with the Lorentzian metric
\[\hbox{$\langle  x,y\rangle= -x_{0}y_0+\sum_{j=1}^{n+1}x_jy_j=x^t I_{1,n+1} y,$ for all $
	x,y\in\R^{n+4}.$}\] Here  $I_{1,n+3}=diag\left(-1,1,\cdots,1\right). $
Let $\mathcal{C}^{n+3}_+= \lbrace x \in \mathbb{R}^{n+4}_{1} |\langle x,x\rangle=0 , x_0 >0 \rbrace $ be
the forward light cone. Let $Q^{n+2}=\mathcal{C}^{n+3}_+/ \R^+$ be the  projective light cone. For a point $Y\in \mathcal {C}^{n+3}_+$, we denote by $[Y]$ its projection in $Q^{n+2}$. Then
we can identify $S^{n+2}$ with $Q^{n+2}$ by setting $y\in S^{n+2}$ to $[Y=(1,y)]\in  Q^{n+2}$.
Let $y:M\rightarrow \SSS^{n+2}$ be a conformal immersion from a Riemann surface $M$. Let $z$ be a local complex coordinate on $U\subset M$ with $e^{2\omega}=2\langle y_z,y_{\bar{z}}\rangle$. We have a canonical lift $Y=e^{-\omega}(1,y)$ into $\mathcal{C}^{n+3}$ with respect to $z$ since $|Y_z|^2=\frac{1}{2}$.
Moreover, there exists a global bundle decomposition
$M\times \mathbb{R}^{n+4}_{1}=V\oplus V^{\perp}.$
Here $V_p={\rm Span}\{Y,{\rm Re}Y_{z},{\rm Im}Y_{z},Y_{z\bar{z}}\}|_{p} \hbox{ for } p\in M$, and $V^{\perp}|_{p}$ is the orthogonal complement of $V_p$ in $ \mathbb{R}^{n+4}_{1}$. Note that $V_p$ is a 4-dimensional Lorenzian subspace and $V^{\perp}|_{p}$ is an $(n)-$dimensional Euclidean subspace. Denote by $V_{\mathbb{C}}$ and
$V^{\perp}_{\mathbb{C}}$ the complexifications of $V$ and $V^{\perp}$ respectively.
Let $\{Y,Y_{z},Y_{\bar{z}},N\}$  be a  frame of
$V_{\mathbb{C}}$ such that
$\langle N,Y_{z}\rangle=\langle N,Y_{\bar{z}}\rangle=\langle
N,N\rangle=0,\ \langle N,Y\rangle=-1$. Let $D$
be the normal connection on $V_{\mathbb{C}}^{\perp}$, and $\psi\in
\Gamma(V_{\mathbb{C}}^{\perp})$ be an arbitrary section of $V_{\mathbb{C}}^{\perp}$.
Then we have:
\begin{equation}\label{eq-moving}
\left\{\begin {array}{lllll}
Y_{zz}&=&-\frac{s}{2}Y+\kappa,  \\
Y _{z\bar{z}}&=&-\langle \kappa,\bar\kappa\rangle Y+\frac{1}{2}N, \\
N_{z}&=&-2\langle \kappa,\bar\kappa\rangle Y_{z}-sY_{\bar{z}}+2D_{\bar{z}}\kappa, \\
\psi_{z}&=&D_{z}\psi+2\langle \psi,D_{\bar{z}}\kappa\rangle Y-2\langle
\psi,\kappa\rangle Y_{\bar{z}}. \\
\end {array}\right. \ \hbox{ Structure equations. }
\end{equation}
Here $\kappa$ and $s$ are named as \emph{the conformal Hopf differential} and
\emph{the Schwarzian} of $y$ respectively \cite{BPP}.
The integrability conditions are
as follows:
\begin{equation}\label{eq-integ}
\left\{\begin {array}{lllll}
\frac{1}{2}s_{\bar{z}}=3\langle
\kappa,D_z\bar\kappa\rangle +\langle D_z\kappa,\bar\kappa\rangle,\  &  \hbox{ Gauss eq.}  \\
{\rm Im}(D_{\bar{z}}D_{\bar{z}}\kappa+\frac{\bar{s}}{2}\kappa)=0,\  &  \hbox{ Codazzi eq.}  \\
R^{D}_{\bar{z}z}\psi=D_{\bar{z}}D_{z}\psi-D_{z}D_{\bar{z}}\psi =
2\langle \psi,\kappa\rangle\bar{\kappa}- 2\langle
\psi,\bar{\kappa}\rangle\kappa,\  &  \hbox{ Ricci eq.}
\end {array}\right.
\end{equation}
The Willmore energy of $y$ is defined to be
\[W(y)=\frac{i}{2}\int_M|\kappa|^2\dd z\wedge \dd \bar z.\]
Let $H$ and $K$ denote the mean curvature and Gauss curvature of $y$ in $\SSS^{n+2}$ respectively. We have
\[W(y)=\int_M(H^2-K+1)\dd M.\]
Note that in many cases the Willmore energy is also defined as
\[\tilde W(y)=\int_M(H^2+1)\dd M=W(y)+\int_MK\dd M.\]
In particular, for an oriented closed surface $M$ with Euler number $\chi(M)$,
\[\tilde W(y)=W(y)+2\pi\chi(M).\]
For compact surfaces with boundary, to get a conformal invariant functional, one needs to use $W(y)$ instead of $\tilde W(y)$ (See e.g. \cite{AM2010,AM2015,Sch}).

For a surface in hyperbolic space $x:M\rightarrow \hh^{n+2}$, with or without boundary, the conformal invariant Willmore energy is defined to be (See e.g. \cite{AM2010,AM2015,Sch}).
\begin{equation}\label{eq-hn-w-energy}
W(x)=\int_M(H^2-K-1)\dd M.
\end{equation}
By the Gauss equation of $x$ one has
\[H^2-K-1=\frac{1}{2}(S-2H^2),\]
where $S$ is the square of the length of the second fundament form of $x$ (Compare Theorem 1.2 of \cite{AM2010}). For the case of surfaces in $S^{n+2}$, see (1.2) and (2.8) of \cite{Li}.
\vspace{3mm}

It is well-known that Willmore surfaces can be characterized as follows.
\begin{theorem} \cite{Bryant1984}, \cite{Ejiri1988}, \cite{BPP}:
	$y$ is a Willmore surface if and only if the Willmore equation holds
	\begin{equation}\label{eq-willmore}
	D_{\bar{z}}D_{\bar{z}}\kappa+\frac{\bar{s}}{2}\kappa=0;
	\end{equation}
	if and only if  the conformal Gauss map $Gr:M\rightarrow
	Gr_{1,3}(\mathbb{R}^{n+4}_{1})=SO^+(1,n+1)/SO^+(1,3)\times SO(n)$ of $y$ is harmonic. Here $Gr$ is defined as
	\[
	Gr:=Y\wedge Y_{u}\wedge Y_{v}\wedge N=-2i\cdot Y\wedge Y_{z}\wedge
	Y_{\bar{z}} \wedge N.
	\]
\end{theorem}

A local lift of $Gr$ into $SO^+(1,n+3)$ can be chosen as \begin{equation}\label{F}
F:=\left(\frac{1}{\sqrt{2}}(Y+N),\frac{1}{\sqrt{2}}(-Y+N),e_1,e_2,\psi_1,\cdots,\psi_{n}\right): U\rightarrow  SO^+(1,n+1)
\end{equation}with Maurer-Cartan form
\[\alpha=F^{-1}\dd F=\left(
\begin{array}{cc}
A_1 & B_1 \\
-B_1^tI_{1,3} & A_2 \\
\end{array}
\right)\dd z+\left(
\begin{array}{cc}
\bar{A}_1 & \bar{B}_1 \\
-\bar{B}_1^tI_{1,3}& \bar{A}_2 \\
\end{array}
\right)\dd\bar{z},\]
and
\begin{equation}\label{eq-b1} B_1=\left(
\begin{array}{ccc}
\sqrt{2} \beta_1 & \cdots & \sqrt{2}\beta_{n} \\
-\sqrt{2} \beta_1 & \cdots & -\sqrt{2}\beta_{n} \\
-k_1 & \cdots & -k_{n} \\
-ik_1 & \cdots & -ik_{n} \\
\end{array}
\right).   \end{equation} Here $\{\psi_j\}$ is an orthonormal basis of $V^{\perp}$ and
$\kappa=\sum_j k_j\psi_j ,\ D_{\bar{z}}\kappa=\sum_j\beta_j\psi_j,\ k =\sqrt{\sum_j|k_j|^2}.$

\ \\

Finally we recall that for a surface $y$ in $\SSS^4$, it is called isotropic if and only if its Hopf differential satisfies
\[\langle\kappa,\kappa\rangle\equiv0\] (see \cite{Calabi,Ejiri1988,Mon,Mus1}).
This is a conformal invariant condition and it plays important roles in the classification of minimal two-spheres \cite{Calabi} and Willmore two-spheres in $\SSS^4$  \cite{Ejiri1988,Mon,Mus1}. It is well-known that if $y$ is an isotropic surface in $S^4$, then it is Willmore \cite{Ejiri1988}.\vspace{3mm}

\subsection{The  DPW construction of  Willmore surfaces in $\SSS^{n+2}$ via conformal Gauss maps}

\subsubsection{The  DPW construction of harmonic maps}
We will recall the basic theory of the DPW methods(See \cite{DPW,DoWa12} for more details).
Let $G/K$ be a symmetric space defined by the involution $\sigma: G\rightarrow G$, with $G^{\sigma}\supset K\supset(G^{\sigma})_0$, and Lie algebras $\mathfrak{g}=Lie(G)$, $\mathfrak{k}=Lie(K)$. Then
$ \mathfrak{g}=\mathfrak{k}\oplus\mathfrak{p}, \
[\mathfrak{k},\mathfrak{k}]\subset\mathfrak{k}, \
[\mathfrak{k},
\mathfrak{p}]\subset\mathfrak{p},
\ [\mathfrak{p},\mathfrak{p}]\subset\mathfrak{k}.$

Let $f: M\rightarrow G/K$ be a harmonic map. Let $z$ be a complex coordinate on $U\subset M$. Then there exists a frame $F: U\rightarrow G$ of $f$ with Maurer-Cartan form $F^{-1} \dd F= \alpha$. The Maurer-Cartan equation reads
$\dd\alpha+\frac{1}{2}[\alpha\wedge\alpha]=0.$ Decompose it with respect to the Cartan decomposition, we obtain
$ \alpha=\alpha_0+\alpha_1 $ with $\alpha_0\in \Gamma(\mathfrak{k}\otimes T^*M), \
\alpha_1\in \Gamma(\mathfrak{p}\otimes T^*M)$.
Decompose $\alpha_1$ further into the $(1,0)-$part $\alpha_{1}'$ and the $(0,1)-$part $\alpha_{1}''$. Introducing $\lambda\in S^1$, set
\begin{equation}\label{eq-alphaloop}
\alpha_{\lambda}=\lambda^{-1}\alpha_{1}'+\alpha_0+\lambda\alpha_{1}'',  \  \lambda\in  S^1.
\end{equation}
It is well known (\cite{DPW}) that the map  $f:M\rightarrow G/K$ is harmonic  if and only if
\[
\dd  \alpha_{\lambda}+\frac{1}{2}[\alpha_{\lambda}\wedge\alpha_{\lambda}]=0\ \ \hbox{for all}\ \lambda \in  S^1.
\]

\begin{definition}Let $F(z,\lambda)$ be a solution to the equation $ \dd  F(z,\lambda)= F(z, \lambda)\alpha_{\lambda},\ F(0,\lambda)=F(0)
	.$ Then $F(z,\lambda)$ is called the {\em extended frame} of the harmonic map $f$.
	Moreover, \[f(z,\lambda):= F(z,\lambda)\mod K\]
	are harmonic maps in $G/K$ for all $\lambda\in S^1$, called the associated family of $f$.
	Note that $f(z,\lambda)=f$ and $F(z,1)=F(z)$.
\end{definition}

So far we have related harmonic maps with maps into loop groups. Moreover, we need the Iwasawa and Birkhoff decompositions for loop groups.
Let $G^{\mathbb{C}}$ be the complexified Lie group of $G$. Extend $\sigma$ to an inner involution of $G^{\mathbb{C}}$  with  $Fix_{\sigma}G^\C=K^{\mathbb{C}}$. Let $\Lambda G^{\mathbb{C}}_{\sigma}$ be the group of loops in $G^\mathbb{C}$ twisted by $\sigma$. Let $\Lambda^-_*G^{\mathbb{C}}_{\sigma}$ be the group of loops  that extends holomorphically into $\infty$ and take values $I$ at $\infty$.
\begin{theorem}\label{thm-iwasawa}   \cite{DPW},  \cite{DoWa10}
	\begin{enumerate}
		\item (Iwasawa decomposition):
		There exists a closed, connected solvable subgroup $S \subseteq K^\C$ such that
		the multiplication $\Lambda G_{\sigma}^0 \times \Lambda^{+}_S G^{\mathbb{C}}_{\sigma}\rightarrow
		\Lambda G^{\mathbb{C}}_{\sigma}$ is a real analytic diffeomorphism onto the open subset
		$ \Lambda G_{\sigma}^0 \cdot \Lambda^{+}_S G^{\mathbb{C}}_{\sigma} = \mathcal{I}^{\mathcal{U}}_e \subset(\Lambda G^{\mathbb{C}}_{\sigma})^0
		$, with $\Lambda_{S}^+ G^{\mathbb{C}}_{\sigma}:=\{\gamma\in\Lambda^+G^{S}_{\sigma}~|~\gamma|_{\lambda=0}\in S   \}.$
		\item (Birkhoff decomposition):
		The multiplication $\Lambda_{*}^{-} {G}^{\mathbb{C}}_{\sigma}\times
		\Lambda^{+}_{\mathcal{C}} {G}^{\mathbb{C}}_{\sigma}\rightarrow
		\Lambda {G}^{\mathbb{C}}_{\sigma}$ is an analytic  diffeomorphism onto the
		open, dense subset $\Lambda_{*}^{-} {G}^{\mathbb{C}}_{\sigma}\cdot
		\Lambda^{+}_{\mathcal{C}} {G}^{\mathbb{C}}_{\sigma}$ of $\Lambda {G}^{\mathbb{C}}_{\sigma}$ {\em (the big Birkhoff cell)}, with  $\Lambda_{\mathcal C}^+ G^{\mathbb{C}}_{\sigma}:=\{\gamma\in\Lambda^+G^{\mathbb{C}}_{\sigma}~|~\gamma|_{\lambda=0}\in (K^\C)^0 \}.$
	\end{enumerate}
\end{theorem}
The well-known DPW construction for harmonic maps can be stated as follows
\begin{theorem} \label{thm-DPW} \cite{DPW}  Let $\mathbb{D}\subset\mathbb{C}$ be a disk or $\mathbb{C}$  with complex coordinate $z$.
	\begin{enumerate}
		\item Let $f:\mathbb{D}\rightarrow G/K$ denote a harmonic map with an extended frame $F(z,\bar{z},\lambda)\in \Lambda G_{\sigma}$ and $F(0,0,\lambda)=I$. Then there exists a Birkhoff decomposition of $F(z,\bar{z},\lambda)$:
		$F_-(z,\lambda)=F(z,\bar{z},\lambda)  F_+(z,\bar{z},\lambda),$  with $ F_+$ taking values in $\Lambda^+_{S}G^{\mathbb{C}}_{\sigma},
		$
		such that $F_-(z,\lambda):\mathbb{D} \rightarrow\Lambda^-_*G^{\mathbb{C}}_{\sigma}$ is meromorphic. Moreover, the Maurer-Cartan form of $F_-$ is the form\[
		\eta=F_-^{-1} \dd  F_-=\lambda^{-1} \eta_{-1}(z) \dd z,
		\]
		called {\em the normalized potential} of $f$, with  $\eta_{-1} :\mathbb{D} \rightarrow \mathfrak{p}\otimes\C$ independent of $\lambda$.
		
		\item Let $\eta$ be a $\lambda^{-1}\cdot\mathfrak{p}\otimes\C-$valued meromorphic 1-form on $\mathbb{D}$. Let $F_-(z,\lambda)$ be a solution to $F_-^{-1} \dd  F_-=\eta$, $F_-(0,\lambda)=I$. Then there exists an Iwasawa decomposition
		\[
		F_-(0,\lambda)=\tilde{F}(z, \bar{z},\lambda)  \tilde{F}^+(z, \bar{z},\lambda),\]
		with  $\tilde{F}\in\Lambda G_{\sigma},\ \tilde{F}\in\Lambda ^+_{S} G^{\mathbb{C}}_{\sigma}$
		on an open subset $\mathbb{D}_{\mathfrak{I}}$ of $\mathbb{D}$. Moreover,  $\tilde{F}(z,\bar{z},\lambda)$ is an extended frame of some harmonic map from $\mathbb{D}_{\mathfrak{I}}$  to $G/K$ with $\tilde{F}(0,0,\lambda)=I$. All harmonic maps can be obtained in this way, since the above two procedures are inverse to each other if the normalization at some based point is fixed.
	\end{enumerate}\end{theorem}

	\subsubsection{Normalized potentials of  Willmore surfaces in $\SSS^{n+2}$}
	For simplicity let us restrict to the case for Willmore surfaces \cite{DoWa10,DoWa12,Wang-2}. In this case, $G=SO^+(1,n+3)$, $K=SO^+(1,3)\times SO(n)$, and $\mathfrak{g}=\mathfrak{so}(1,n+3)=\{X\in \mathfrak{g}l(n+4,\mathbb{R})|X^tI_{1,n+3}+I_{1,n+3}X=0\}.$
	The involution $\sigma$ is given by $ \sigma: \ SO^+(1,n+3)  \rightarrow \ SO^+(1,n+3), \sigma(A):=DAD^{-1},$
	with  $D=\hbox{diag}\{-I_{4}, I_{n}\}$. We also have $\mathfrak{g}=\mathfrak{k}\oplus\mathfrak{p},$
	with
	\[ \mathfrak{k}=\left\{\left(
	\begin{array}{cc}
	A_1 &0 \\
	0 & A_2 \\
	\end{array}
	\right)
	|A_1^tI_{1,3}+I_{1,3}A_1 =0, A_2+A_2^t=0\right\}, \  \mathfrak{p}=\left\{\left(
	\begin{array}{cc}
	0 & B_1 \\
	-B_1^tI_{1,3} & 0 \\
	\end{array}
	\right)
	\right\}.
	\]
	Let $G^{\mathbb{C}}=SO^+(1,n+3,\mathbb{C})=  \{X \in SL(n+4,\C) ~|~ X^t I_{1,n+3} X=I_{1,n+3}\}$ with Lie algebra $\mathfrak{so}(1,n+3,\mathbb{C})$. Extend $\sigma$ to an inner involution of $SO^+(1,n+3,\mathbb{C})$  with  fixed point group $K^{\mathbb{C}}=S(O^+(1,3,\mathbb{C})\times O(n,\C))$.
	
	Since Willmore surfaces and their oriented conformal Gauss map are in one to one correspondence \cite{DoWa10, Ejiri1988,Ma}, we will use the normalized potential for a Willmore surface directly.
	For later use, we recall the description of minimal surfaces in space forms in terms of normalized potentials.
	\begin{theorem}\label{thm-minimal}  \cite{Wang-2} (compare also \cite{BW,Helein,Xia-Shen})
		Let $y$ be a Willmore surface in $\SSS^{n+2}$, with its normalized potential being of the form
		\[\eta=\lambda^{-1}\eta_{-1}\dd z=\lambda^{-1}\left(
		\begin{array}{cc}
		0 & \hat{B}_1 \\
		-\hat{B}_1^tI_{1,3} & 0 \\
		\end{array}
		\right)\dd z,\ \hbox{ and }\hat{B}_1^tI_{1,3}\hat B_1=0.\]
		Then
		$y$ is conformally equivalent to some minimal surface in $\R^{n+2}$, $\SSS^{n+2}$ or $\mathbb H^n$ if and only if there exists a non-zero, real, constant vector $\mathbf{v}=(\mathrm{v}_1,\mathrm{v}_2,\mathrm{v}_3,\mathrm{v}_4)^t\in \R^4_1$ such that
		\begin{equation}\label{eq-min}
		\mathbf{v}^tI_{1,3}\hat B_1\equiv0.
		\end{equation}
		Moreover,
		\begin{enumerate}
			\item  the space form is $\R^{n+2}$ if and only if $\langle\mathbf{v},\mathbf{v}\rangle=\mathbf{v}^tI_{1,3}\mathbf{v}=0;$
			\item the space form is $\SSS^{n+2}$ if and only if $\langle\mathbf{v},\mathbf{v}\rangle=\mathbf{v}^tI_{1,3}\mathbf{v}<0;$
			\item the space form is $\mathbb H^{n+2}$ if and only if $\langle\mathbf{v},\mathbf{v}\rangle=\mathbf{v}^tI_{1,3}\mathbf{v}>0.$
		\end{enumerate}
	\end{theorem}
	
	Note that in \cite{DIK}, \cite{BHS}, there are some different treatments of minimal surfaces in $\hh^3$ via loop group methods.
	\section{$K^{\C}-$dressing actions on Willmore surfaces}
	
	In this section, we will use the dressing actions on harmonic maps by $K^{\C}$ for Willmore surfaces. We refer to \cite{BQ, Gu-Oh, LM, TU, Uh} for more details on dressing actions and their applications on all kinds of geometric problems. Note that here we use the elements in $K^\C$ instead of the loop group elements.
	
	\subsection{$K^{\C}-$dressing actions}
	\begin{definition}\label{def-dress}
		Let $\mathbf{k}\in K^{\C}$. Let  $f:\mathbb D\rightarrow G/K$  be a harmonic map with an extended frame $F(z,\lambda)$, based at $z_0$ such that $F(z_0,\lambda)=e\in G$. a dressing action by $\mathbf{k}$ on $f$ is defined by the harmonic map \[\mathbf{k}\sharp f:=\hat{F}\mod K,\] where $\hat F:\D\rightarrow \Lambda G_{\sigma}$ is given by the following
		\begin{equation}\label{eq-dressing-k}
		\hat F= \mathbf{k} F(z,\lambda) \hat{V}_+, \hbox{ with } \hat{V}_+\in \Lambda^+ G^{\C}_{\sigma}.\end{equation}
	\end{definition}
	From the definition it is obvious that
	\begin{corollary} $\tilde{f}=\mathbf k\sharp f$ if  $f=\mathbf k^{-1}\sharp \tilde{f}$.
	\end{corollary}
	
	The following result is well-known to the experts. For the reader's convenience, we state it in the following way with a proof.
	\begin{proposition} Let $\eta$ and $\hat\eta$ be the normalized potentials of $f$ and $\mathbf k\sharp f$ given by the extended frames $F$ and $\hat F$ respectively. Then
		\begin{equation}\label{eq-dressing-k-np}
		\hat\eta=\mathbf k \eta \mathbf{k}^{-1}.
		\end{equation}
		
		Conversely, assume that $\eta$ and $\hat\eta$ satisfies \eqref{eq-dressing-k-np}, and their integrations have the same initial conditions, then their corresponding harmonic maps $f$ and $\hat f$ satisfy $\hat f=\mathbf k \sharp f$.
	\end{proposition}
	
	\begin{proof} By Theorem \ref{thm-DPW}, we have
		\[F_-=F F_+ \hbox{ and } \eta=F_-^{-1}\dd F_-,~ \hat F_-=\hat F\hat F_+ \hbox{ and  }\hat\eta=\hat F_-^{-1}\dd\hat F_-.\]
		From \eqref{eq-dressing-k}, we also have $ \hat F= \mathbf{k} F(z,\lambda) \hat{V}_+$. So
		\[\hat F_-=\mathbf{k} F(z,\lambda) \hat{V}_+ \hat F_+=\mathbf{k} F_- F_+^{-1}\hat{V}_+=\mathbf{k} F_-\mathbf{k}^{-1} \mathbf k F_+^{-1}\hat{V}_+.\]
		Together with the assumption of having same initial conditions, we obtain that $\hat F_-=\mathbf{k} F_-\mathbf{k}^{-1} $, and \eqref{eq-dressing-k-np} follows directly.
		
		Concerning the converse part, first by assumptions we have $\hat F_-=\mathbf{k} F_-\mathbf{k}^{-1} $. So
		\[\hat F=\hat F_- \hat F_+^{-1}= \mathbf{k} F_-\mathbf{k}^{-1} \hat F_+^{-1}=\mathbf{k} F \hat V_+\]
		with $\hat V_+=F_+^{-1}\mathbf{k}^{-1} \hat F_+^{-1},$ that is, $\hat f=\mathbf k\sharp f.$
	\end{proof}
	
	Applying to Willmore surfaces, we obtain
	\begin{proposition}\label{prop-dr}
		Let $T=T_1\times T_2\in  SO(1,3,\C)\times SO(n,\C)$. Let $f$ be a harmonic map with normalized potential
		\begin{equation}\label{eq-w-np}
		\eta=\lambda^{-1}\eta_{-1}dz=\lambda^{-1}\left(
		\begin{array}{cc}
		0 & \hat{B}_1 \\
		-\hat{B}_1^tI_{1,3} & 0 \\
		\end{array}
		\right)dz, \hbox{ with } \hat{B}_1^tI_{1,3}\hat B_1=0.
		\end{equation}
		Then the normalized potential $\eta_{T}$ of $T\sharp f$ has the form
		\begin{equation}\label{eq-action}
		\eta_{T}=T\eta T^{-1}=\lambda^{-1}\left(
		\begin{array}{cc}
		0 & T_1\hat{B}_1T_2^t \\
		-T_2\hat{B}_1^tI_{1,3}T_1^t & 0 \\
		\end{array}
		\right)dz.
		\end{equation}
	\end{proposition}
	\ \\
	
	We define  the space of the conformal Gauss maps of minimal surfaces in three space forms
	\begin{equation}\label{eq-minimal-set-r}
	\mathcal{M}_0:=\{f| f \hbox{ is the conformal Gauss map of a minimal surface in } \R^{n+2}\},
	\end{equation}
	\begin{equation}\label{eq-minimal-set-s}
	\mathcal{M}_1:=\{f| f \hbox{ is the conformal Gauss map of a minimal surface in } \SSS^{n+2}\},
	\end{equation}
	\begin{equation}\label{eq-minimal-set-h}
	\mathcal{M}_{-1}:=\{f| f \hbox{ is the conformal Gauss map of a minimal surface in } \hh^{n+2}\}.
	\end{equation}
	We also define the space $\mathcal{M}_L$ and its subset $\widetilde{\mathcal{M}}_0$ as below
	\begin{equation}\label{eq-np-Li}
	\mathcal{M}_{L}:=\left\{f| \hbox{The normalized potential $\eta$ of $f$ satisfies } \mathbf{v}^tI_{1,3}\hat B_1=0  \hbox{ for some }\mathbf v\in\C^4_1\backslash\{0\}\right\},
	\end{equation}\begin{equation}\label{eq-np-Li-s}
	\widetilde{\mathcal{M}}_0:=\left\{f| f\in \mathcal{M}_{L}, \hbox{ with  $\mathbf v$ satisfying }  \mathbf{v}^tI_{1,3}\mathbf v=0\right\}.
	\end{equation}
	Note that $\mathcal{M}_0\subsetneqq \widetilde{\mathcal{M}}_0$ (See \cite{Wang-2} for example). In \cite{Wang-2}, it is shown that up to a conjugation, for any $f\in\widetilde{\mathcal{M}}_0$, the normalized potential of $f$ has the form ((1) of \cite{Wang-2})
	\[\hat B_1=
	\left(
	\begin{array}{ccccccc}
	\hat f_{11} & \hat f_{12} &  \cdots &  \hat{f}_{1,n} \\
	-\hat f_{11} & -\hat f_{12} &  \cdots &  -\hat{f}_{1,n} \\
	\hat f_{13} & \hat f_{32} &  \cdots &  \hat{f}_{3n} \\
	i\hat f_{13} & i\hat f_{32} &  \cdots &  i\hat{f}_{3n} \\
	\end{array}
	\right)\]
	with $\hat f_{ij}$ being meromorphic functions.

	Set $K^{\C}=SO(1,3,\C)\times SO(n,\C)$ and we define
	\begin{equation}\label{eq-KC-M}
	\begin{split}
	K^{\C}\sharp \mathcal{M}_{j}&:=\{T\sharp f|T\in K^{\C}, f\in \mathcal{M}_j\},\ j=0,1,-1;\\
	K^{\C}\sharp \widetilde{\mathcal{M}}_{0}&:=\{T\sharp f|T\in K^{\C}, f\in \widetilde{\mathcal{M}}_0\};\\
	K^{\C}\sharp \mathcal{M}_{L}&:=\{T\sharp f|T\in K^{\C}, f\in \mathcal{M}_L\}.\\
	\end{split}
	\end{equation}

	\subsection{$K^{\C}-$dressing actions preserve minimal surfaces in $\R^{n+2}$}
	
	\begin{theorem} \label{th-dress-rn} Let $f$ be the oriented conformal Gauss map of a minimal surface in $\R^{n+2}$ and $T\in K^{\C}$. Then $T\sharp f$ is also the oriented conformal Gauss map of a minimal surface in $\R^{n+2}$, i.e.,
		\[K^{\C}\sharp\mathcal{M}_0={\mathcal{M}}_0.\]
	\end{theorem}
	
	\begin{proof}
		To show that $  K^{\C}\sharp\mathcal{M}_0=\mathcal{M}_0$, we need to show that $K^{\C}\sharp\mathcal{M}_0\setminus \mathcal{M}_0=\emptyset$. Otherwise assume that $\tilde f\in K^{\C}\sharp\mathcal{M}_0\setminus \mathcal{M}_0$. Then there exists $f\in\mathcal{M}_0$ and $T\in K^{\C}$  such that
		$T\sharp f=\tilde f$. So $T^{-1}\sharp \tilde f=f$. Assume that $T=\hbox{diag}(T_1,T_2)$ and the normalized potential of $f$ is given by $\hat B_1$.
		Then the normalized potential of $\tilde f$ is given by $T_1\hat B_1 T_2^{-1}$ by \eqref{eq-action}.
		
		Since $\tilde f\not\in\mathcal{M}_0$, by \cite{Wang-2} we have that either   $T_1\hat B_1 T_2^{-1}$ has rank $2$ or $\tilde f$ reduces to  a map into $SO(n+2)/SO(2)\times SO(n)$ or $SO(1,1)/SO(1,1)\times SO(n)$. If  $\tilde B_1$ has maximal rank 2, then $\hat B_1$ also has maximal rank 2, which is not possible since $\hat B_1$ has maximal rank 1  due to the assumption $f\in\mathcal{M}_0$. If $\tilde f$ reduces to  a map into $SO(n+2)/SO(2)\times SO(n)$ or $SO(1,1)/SO(1,1)\times SO(n)$, then we can assume w.l.g. that
		\[T_1\hat B_1 T_2^{-1}=\left(
		\begin{array}{ccccccc}
		\tilde f_{11} & \tilde f_{12} &  \cdots &  \tilde{f}_{1,n} \\
		-\tilde f_{11} & -\tilde f_{12} &  \cdots &  -\tilde{f}_{1,n} \\
		0 & 0 &  \cdots &  0 \\
		0 & 0 &  \cdots &  0 \\
		\end{array}
		\right) \hbox{ or }\left(
		\begin{array}{ccccccc}
		0 & 0 &  \cdots &  0 \\
		0 & 0 &  \cdots &  0 \\
		\tilde f_{13} & \tilde f_{32} &  \cdots &  \tilde{f}_{3n} \\
		i\tilde f_{13} & i\tilde f_{32} &  \cdots &  i\tilde{f}_{3n} \\
		\end{array}
		\right).\]
		So
		\[\hat B_1=T_1^{-1}\left(
		\begin{array}{ccccccc}
		\tilde f_{11} & \tilde f_{12} &  \cdots &  \tilde{f}_{1,n} \\
		-\tilde f_{11} & -\tilde f_{12} &  \cdots &  -\tilde{f}_{1,n} \\
		0 & 0 &  \cdots &  0 \\
		0 & 0 &  \cdots &  0 \\
		\end{array}
		\right)T_2 \hbox{ or }\hat B_1=T_1^{-1}\left(
		\begin{array}{ccccccc}
		0 & 0 &  \cdots &  0 \\
		0 & 0 &  \cdots &  0 \\
		\tilde f_{13} & \tilde f_{32} &  \cdots &  \tilde{f}_{3n} \\
		i\tilde f_{13} & i\tilde f_{32} &  \cdots &  i\tilde{f}_{3n} \\
		\end{array}
		\right)T_2.\]
		Consider in the first case the constant vector
		\[\mathbf v^*=T_1^{-1}\left(
		\begin{array}{ccccccc}
		1  \\
		-1\\
		0\\
		0\\
		\end{array}
		\right)\in\C^4_1.\] Apparently  the action of $T_2$ does not change its form. So we can assume without lose of generality $T_2=I$.
		Note by construction $\mathbf{v}^*$ stays an isotropic vector, i.e., $(\mathbf{v}^*)^tI_{1,3}\mathbf{v}^*=0$. So there exists some real $\check T_1\in SO(1,3)$ such that
		\[\check T_1T_1^{-1}\left(
		\begin{array}{ccccccc}
		1  \\
		-1\\
		0\\
		0\\
		\end{array}
		\right)=a\left(
		\begin{array}{ccccccc}
		1  \\
		-1\\
		0\\
		0\\
		\end{array}
		\right)\hbox{ or }a\left(
		\begin{array}{ccccccc}
		0  \\
		0\\
		1\\
		i\\
		\end{array}
		\right),\]
		depending on  whether $Re \mathbf{v}^*$ and $Im \mathbf{v}^*$ are linear dependent or not.
		Here $a$ is a constant number. So up to an action $\check T=\hbox{diag}(\check T_1,I_n)$, $f$ reduces to a harmonic map into $SO(1,n+1)/SO(1,1)\times SO(n)$, which is contradicted to the assumption $f\in\mathcal {M}_0$, since harmonic maps into $SO(1,n+1)/SO(1,1)\times SO(n)$ does not produce minimal surfaces in $\R^{n+2}$. Similarly, the second case produces a harmonic map into $SO(n+2)/SO(2)\times SO(n)$,
		which also does not give minimal surfaces in $\R^{n+2}$. Hence $K^{\C}\sharp\mathcal{M}_0= \mathcal{M}_0.$
	\end{proof}
	\begin{remark} In \cite{LM}, it is shown that the simple dressing actions preserve minimal surfaces in $\R^4$. Our result here shows that  $K^{\C}$ dressing actions   preserve minimal surfaces in $\R^{n+2}$.
	\end{remark}

	\subsection{$K^{\C}-$dressing actions on minimal surfaces in $\SSS^{n+2}$ and $\hh^{n+2}$}
	\begin{theorem} \label{th-dress-sn} \
		
		\begin{enumerate}
			\item Let $f\in \mathcal{M}_1$. Then $T\sharp f\in {\mathcal{M}}_L\setminus\widetilde{\mathcal{M}}_0$ for any $T\in K^{\C}$.
			Conversely, let $\tilde f\in {\mathcal{M}}_L\setminus\widetilde{\mathcal{M}}_0$. Then there exists some $\tilde T\in K^{\C}$ such that $\tilde T\sharp \tilde f\in\mathcal{M}_1$. That is
			\begin{equation}\label{eq-K-M1}
			K ^{\C}\sharp\mathcal{M}_1={\mathcal{M}}_L\setminus\widetilde{\mathcal{M}}_0.
			\end{equation}
			\item Let $f\in \mathcal{M}_{-1}$. Then $T\sharp f\in {\mathcal{M}}_L\setminus\widetilde{\mathcal{M}}_0$ for any $T\in K^{\C}$.
			Conversely, let $\tilde f\in {\mathcal{M}}_L\setminus\widetilde{\mathcal{M}}_0$. Then there exists some $\tilde T\in K^{\C}$ such that $\tilde T\sharp \tilde f\in\mathcal{M}_{-1}$. That is
			\begin{equation}\label{eq-K-M-1}K^{\C}\sharp\mathcal{M}_{-1}={\mathcal{M}}_L\setminus\widetilde{\mathcal{M}}_0.\
			\end{equation}
			\item  In particular, for any $f\in \mathcal{M}_{1}$, there exists some $T\in K^{\C}$ such that $\hat f=T\sharp f\in \mathcal{M}_{-1}$.
			For any $f\in\mathcal{M}_{-1}$, there exists  $T\in K^{\C}$ such that $\hat f=T\sharp f\in \mathcal{M}_{1}$.  That is,
			\[\mathcal{M}_{-1}\subseteq K^{\C}\sharp \mathcal{M}_{1},\ \mathcal{M}_{1}\subseteq K^{\C}\sharp \mathcal{M}_{-1}.\]\end{enumerate}
	\end{theorem}
	\begin{proof} (1) Let $f\in\mathcal{M}_{1}$ with normalized potential given by $\hat B_1$. Then by Theorem \ref{thm-minimal}, there exists some $\mathbf v\in\R^4_1$ such that
		\[\mathbf{v}^tI_{1,3}\hat B_1\equiv0,\ \langle\mathbf{v},\mathbf{v}\rangle=\mathbf{v}^tI_{1,3}\mathbf{v}<0.\] So for any $T\in K^\C$,
		the normalized potential of $T\sharp f$ is given by
		$\tilde B_1=T_1\hat B_1T_2^{-1}$, where $T=\hbox{diag}(T_1,T_2)$.  So
		$\tilde{\mathbf v}=T_1\mathbf v$ is the vector such that $\tilde{\mathbf v}^tI_{1,3}T_1\hat B_1T_2^{-1}=0$. So $T\sharp f\in \mathcal{M}_L$.
		Since $\tilde{\mathbf v}^tI_{1,3}\tilde{\mathbf v}={\mathbf v}^tI_{1,3}{\mathbf v}<0$, $T\sharp f \in {\mathcal{M}}_L\setminus\widetilde{\mathcal{M}}_0$.

		Now let $\tilde f\in {\mathcal{M}}_L\setminus\widetilde{\mathcal{M}}_0$ with normalized potential given by  $\tilde B_1$. Then there exists some $\mathbf v$ such that
		$\mathbf{v}^tI_{1,3}\tilde B_1\equiv0,\ \langle\mathbf{v},\mathbf{v}\rangle=\mathbf{v}^tI_{1,3}\mathbf{v}\neq0.$  Set
		$\mathbf v_1=\frac{i}{\sqrt{\mathbf{v}^tI_{1,3}\mathbf v}}\mathbf{v}$. There exists some $\mathbf v_j$, $j=2,3,4$ such that
		\[T_1=
		\left( \begin{array}{ccccc}
		\mathbf v_1 & \mathbf v_2 & \mathbf v_3 & \mathbf v_4
		\end{array}\right)^t\in SO(1,3,\C)\]
		Set $\tilde T=\hbox{diag}(T_1,I_n)$ and \[\mathbf v_0=T_1\mathbf v_1=\left( \begin{array}{ccccc}
		-1 & 0 & 0 &  0
		\end{array}\right)^t.\] We see that $\tilde T\sharp\tilde f$ has normalized potential $\hat B_1=T_1\tilde B_1$ such that
		\[\mathbf v_0^tI_{1,3}\hat B_1=\mathbf v_0^tT_1^tI_{1,3}T_1\hat B_1=\mathbf v_0^tT_1^tI_{1,3}\tilde B_1=\mathbf{v}_1^tI_{1,3}\tilde B_1=\frac{i}{\sqrt{\mathbf{v}^tI_{1,3}\mathbf v}}\mathbf{v}^tI_{1,3}\tilde B_1=0.\]
		Since $\mathbf v_0\in\R^4_1$ and $\mathbf v_0^tI_{1,3}\mathbf v_0<0$, $\tilde T\sharp \tilde f\in \mathcal{M}_{1}$.
		
		The proof of (2) is similar to (1) and we leave it for interested readers.
		(3) is a corollary of (1) and (2).  \end{proof}
	
	\section{On $K^\C-$dressing actions of minimal surfaces in $\SSS^{n+2}$ $\&$ $\hh^{n+2}$  }
	
	We will first discuss the general $K^\C-$dressing actions briefly. Then we will consider concretely two kinds of 1-parameter subgroups of $SO(1,3,\C)$ and their actions on dressing actions of minimal surfaces in $\SSS^{n+2}$  and $\hh^{n+2}$. One of the group changes the minimality and builds a local Willmore deformation between minimal surfaces in $\SSS^{n+2}$  and $\hh^{n+2}$. And the other one keeps  the minimality  and gives a family of minimal surfaces in $\SSS^{n+2} $($\hh^{n+2}$).
	
	\subsection{ $K^\C-$dressing actions of minimal surfaces in $\SSS^{n+2}$  \& $\hh^{n+2}$ }
	
	It is direct to have the following proposition by Theorem \ref{thm-minimal} and Proposition \ref{prop-dr}.
	\begin{proposition}\
		\begin{enumerate}
			\item
			The dimension of non-trivial $K^\C-$dressing actions of a Willmore surfaces in $\SSS^{n+2}$ is less or equal to
			\begin{equation}\label{eq-dim-KC}
			\dim SO(1,3,\C)\times SO(n,\C)-\dim SO(1,3)\times SO(n) =\frac{n(n-1)}{2}+6.
			\end{equation}
			\item
			The dimension of non-trivial $K^\C-$dressing actions of a minimal surface in $\SSS^{n+2}$ preserving minimality infinitesimally  is less or equal to
			\begin{equation}\label{eq-dim-KC-m1}
			\dim SO(3,\C)\times SO(n,\C)-\dim SO(3)\times SO(n) =\frac{n(n-1)}{2}+3.
			\end{equation}
			\item
			The dimension of non-trivial $K^\C-$dressing actions of a minimal surface in $\hh^{n+2}$ preserving minimality infinitesimally  is less or equal to
			\begin{equation}\label{eq-dim-KC-m2}
			\dim SO(1,2,\C)\times SO(n,\C)-\dim SO(1,2)\times SO(n) =\frac{n(n-1)}{2}+3.
			\end{equation}
		\end{enumerate}
	\end{proposition}
	The above spaces of the non-trivial $K^\C-$dressing actions can be locally expressed
	(near $I$) as  $\exp \mathfrak{S}, $ $\exp \mathfrak{S}_1$ and $\exp \mathfrak{S}_{-1} $ respectively, where
	\[\begin{split}
	\mathfrak{S}&=\{A\in \mathfrak{so}(1,3,\C)\times \mathfrak{so}(n,\C)|A=-\bar A\}, \\
	\mathfrak{S}_1&=\{A\in \mathfrak{so}(3,\C)\times \mathfrak{so}(n,\C)|A=-\bar A\},  \\
	\mathfrak{S}_{-1}&=\{A\in \mathfrak{so}(1,2,\C)\times \mathfrak{so}(n,\C)|A=-\bar A\}.\\
	\end{split} \]
	Here $\mathfrak{so}(3,\C)$ and $\mathfrak{so}(1,2,\C)$ are viewed as subsets of $\mathfrak{so}(1,3,\C)$ naturally.
	
	\subsection{On some $S^1-$dressing actions of minimal surfaces in $\SSS^{n+2}$  \& $\hh^{n+2}$ }
	In this subsection, we discuss a special $S^1-$dressing actions which build a smooth local Willmore deformations between  minimal surfaces in $\SSS^{n+2}$ and $\hh^{n+2}$.
	Set
	\begin{equation*}
	T_{1,t}=\left(
	\begin{array}{cccc}
	\cos t& i\sin t & 0& 0 \\
	i\sin t & \cos t & 0& 0 \\
	0& 0 & 1& 0 \\
	0& 0 & 0& 1 \\
	\end{array}
	\right)\in SO(1,3,\C),\ t\in [0,2\pi].
	\end{equation*}
	Then $T_{1,t}$, $t\in [0,2\pi]$, is a circle subgroup of $SO(1,3,\C)$.
	We see that $T_{1,t}\in  SO(1,3)\cap SO(1,3,\C)$ if and only if $t=0,\pi, 2\pi$. And $T_{1,t}\in ( i\cdot O(1,3))\cap SO(1,3,\C)$ if and only if $t=\frac{\pi}{2}, \frac{3\pi}{2}$.
	
	First, assume without lose of generality that the normalized potential of  a Willmore surface in $\SSS^{n+2}$ has the form
	\[
	\eta=\lambda^{-1}\left(
	\begin{array}{cc}
	0 & \hat{B}_1 \\
	-\hat{B}_1^tI_{1,3} & 0 \\
	\end{array}
	\right)dz, \hat{B}_1=\left(\begin{array}{ccccc}
	\hhh_1 & \cdots & \hhh_{n}
	\end{array}  \right).   \]
	By Theorem \ref{thm-minimal}, we can assume without lose of generality that the normalized potential of a minimal surface  in $\SSS^{n+2}$   has the form
	\begin{equation}\label{eq-min-sn-p}
	\hhh_j=h_{0j}\left(
	\begin{array}{c}
	0 \\
	\widetilde{ h_{2}} \\
	\widetilde{ h_{3}} \\
	\widetilde{ h_{4} }\\
	\end{array}
	\right),\ (\widetilde{h_{2}})^2+(\widetilde{h_{3}})^2+(\widetilde{h_{4}})^2=0, j=1,\cdots,n.
	\end{equation}
	Here $h_2$, $h_3$ and $h_4$ are linear independent meromorphic functions.
	\begin{theorem}\label{th-min-sn}
		The normalized potential
		\begin{equation*}%\label{eq-s1-1}
		\eta_t=\lambda^{-1}\left(
		\begin{array}{cc}
		0 & T_{1,t}\hat{B}_1 \\
		-\hat{B}_1^tT_{1,t}^tI_{1,3} & 0 \\
		\end{array}
		\right)dz,
		\end{equation*}
		with $ \hat{B}_1=\left(\begin{array}{ccccc}
		\hhh_1 & \cdots & \hhh_{n}
		\end{array}  \right)$ and   all of $\{\hhh_j,1\leq j\leq n\}$ being of the form \eqref{eq-min-sn-p}, locally gives a family of Willmore surfaces $y_t$, $t\in[0,2\pi)$,   such that  $(y_t)|_{t=0}$, $(y_t)|_{t=\pi}$ are conformally equivalent to  minimal surfaces in $\SSS^{n+2}$ and  $(y_t)|_{t=\frac{\pi}{2}}$, $(y_t)|_{t=\frac{3\pi}{2}}$ are conformally equivalent to  minimal surfaces in $\hh^{n+2}$, and for all other $t$, $y_t$ are Willmore surfaces in $\SSS^{n+2}$  not minimal in any space forms.
	\end{theorem}
	\begin{proof}
		It is direct to see that
		\[\mathbf v_t=\left( \begin{array}{ccccc}
		\cos t& i\sin t & 0 &  0
		\end{array}\right)^t=T_{1,t}^t\left( \begin{array}{ccccc}
		1 & 0 & 0 &  0
		\end{array}\right)^t\]
		satisfies
		\[\mathbf v_t^tI_{1,3}T_{1,t}\hat B_1\equiv0.\]
		So when $t=0$ or $\pi$, one obtains minimal surfaces in $\SSS^{n+2}$.
		So when $t=\frac{\pi}{2}$ or $\frac{3\pi}{2}$, one obtains minimal surfaces in $\hh^{n+2}$.
		
		For other $t$, assume $y_t$ is conformal to some minimal surface in space forms. Then there exists a real vector $\mathbf v\in\R^4_1$ such that
		$\mathbf vI_{1,3}T_{1,t}\hat B_1\equiv0$. So w.l.g. we can assume $\mathbf v=\left( \begin{array}{ccccc}
		a & b& c&  0
		\end{array}\right)^t$.  So we have
		\[-ia h_2\sin t+bh_2\cos t+c h_3 =0\]
		Since $a,b\in \R$, we see that $c\neq0$ and $h_3=-\frac{b-ia}{c}h_2$, which contradicts to the fact that $h_2$ and $h_3$ are linear independent. This finishes the proof.
	\end{proof}
	Similarly, by Theorem \ref{thm-minimal}  we can assume without lose of generality that the normalized potential of a minimal surface  in  $\hh^{n+2}$ has the form
	\begin{equation}\label{eq-min-hn-p}
	\hhh_j= h_{0j}\left(
	\begin{array}{c}
	\widetilde{ h_{1}} \\
	0 \\
	\widetilde{ h_{3}} \\
	\widetilde{ h_{4} }\\
	\end{array}
	\right),~-(\widetilde{h_{1}})^2+(\widetilde{h_{3}})^2+(\widetilde{h_{4}})^2=0, j=1,\cdots,n.
	\end{equation}
	\begin{theorem}\label{th-min-hn}The normalized potential
		\begin{equation*}%\label{eq-s1-1}
		\eta_t=\lambda^{-1}\left(
		\begin{array}{cc}
		0 & T_{1,t}\hat{B}_1 \\
		-\hat{B}_1^tT_{1,t}^tI_{1,3} & 0 \\
		\end{array}
		\right)dz,
		\end{equation*}
		with $ \hat{B}_1=\left(\begin{array}{ccccc}
		\hhh_1 & \cdots & \hhh_{n}
		\end{array}  \right)$ and all of $\{\hhh_j,1\leq j\leq n\}$ being of the form \eqref{eq-min-hn-p},
		locally gives a family of Willmore surfaces $y_t$, $t\in[0,2\pi)$,  such that $(y_t)|_{t=0}$, $(y_t)|_{t=\pi}$ are conformally equivalent to  minimal surfaces in $\hh^{n+2}$ and $(y_t)|_{t=\frac{\pi}{2}}$, $(y_t)|_{t=\frac{3\pi}{2}}$ are conformally equivalent to  minimal surfaces in $\SSS^{n+2}$, and for all other $t$, $y_t$ is a non-minimal Willmore surface in $\SSS^{n+2}$.
	\end{theorem}
	\begin{proof} The proof is the same as above theorem. So we omit it.
	\end{proof}
	\begin{remark} Comparing  \eqref{eq-min-sn-p} and  \eqref{eq-min-hn-p}, we see that the loop group data of minimal surfaces in $\SSS^{n+2}$ and $\hh^{n+2}$ differ essentially by some shifting and multiplying some $i$ for some terms, which can achieved of the above dressing action. This is the key observation \& motivation of the   $K^{\C}$ dressing action.
	\end{remark}
	\subsection{On some $\R^1-$dressing actions preserving minimal surfaces in $\SSS^{n+2}$ \&  $\hh^{n+2}$}
	Set
	\begin{equation*}
	T_{2,t}=\left(
	\begin{array}{cccc}
	1& 0& 0 & 0\\
	0& 1& 0 & 0 \\
	0& 0& \cosh t& i\sinh t \\
	0& 0& -i\sinh t & \cosh t  \\
	\end{array}
	\right)\in SO(1,3,\C),\ t\in \R.
	\end{equation*}
	Then $T_{2,t}$, $t\in\R$, is a  $\R-$subgroup of $SO(1,3,\C)$.
	Note that $T_{2,t}\in  SO(1,3)\cap SO(1,3,\C)$ if and only if $t=0$.
	\begin{theorem}\
		\begin{enumerate}
			\item The normalized potential
			\begin{equation*}\label{eq-r1-1}
			\eta_t=\lambda^{-1}\left(
			\begin{array}{cc}
			0 & T_{2,t}\hat{B}_1 \\
			-\hat{B}_1^tT_{2,t}^tI_{1,3} & 0 \\
			\end{array}
			\right)dz,
			\end{equation*}
			with $ \hat{B}_1=\left(\begin{array}{ccccc}
			\hhh_1 & \cdots & \hhh_{n}
			\end{array}  \right)$ and $\hhh_j$ being of the form \eqref{eq-min-sn-p}, locally gives a family of Willmore surfaces $y_t$ conformally equivalent to  minimal surfaces in $\SSS^{n+2}$.
			\item
			The normalized potential
			\begin{equation*} \eta_t=\lambda^{-1}\left(
			\begin{array}{cc}
			0 & T_{2,t}\hat{B}_1 \\
			-\hat{B}_1^tT_{2,t}^tI_{1,3} & 0 \\
			\end{array}
			\right)dz,
			\end{equation*}
			with $ \hat{B}_1=\left(\begin{array}{ccccc}
			\hhh_1 & \cdots & \hhh_{n}
			\end{array}  \right)$ and $\hhh_j$ being of the form \eqref{eq-min-hn-p}, locally gives a family of Willmore surfaces $y_t$ conformally equivalent to  minimal surfaces in $\hh^{n+2}$.
		\end{enumerate}
	\end{theorem}
	\begin{proof}  By Theorem \ref{thm-minimal} and setting  $\mathbf v=\left( \begin{array}{ccccc}
		1 & 0 & 0 &  0
		\end{array}\right)^t$ and $\left( \begin{array}{ccccc}
		0 & 1 & 0 &  0
		\end{array}\right)^t$ respectively, we obtain (1) and (2) respectively.
	\end{proof}
	
	\begin{remark} Note that the $T_{2,t}$ action on \eqref{eq-min-sn-p}, are used exactly as the famous Lopez-Ros deformation for minimal surfaces in $\R^3$ \cite{LR}. We refer to \cite{LM} for the simple factor dressing expression of the Lopez-Ros deformation for minimal surfaces in $\R^3$, which is different from the action considered in this paper.
	\end{remark}

	\section{Examples of Minimal surfaces in $\mathbb{H}^4$}

	In this section, we will illustrate the $K^{\C}-$dressing actions for isotropic minimal surfaces in $\SSS^{4}$ in terms of the formula in above section.
	$K^{\C}-$dressing actions of the Veronese 2-spheres give many explicit examples of Willmore two-spheres in $S^4$. In particular, we obtain many examples of complete minimal surfaces in $\hh^4$, defined on disks, annulus or Moebius strips. By these examples we show that there exists complete minimal disks with their Willmore energy tending to zero.
	Moreover, by consider the Willmore deformations of generalizations of Veronese two-spheres in $\SSS^4$, we obtain  complete minimal disks with arbitrary Willmore energy. Some new non-oriented minimal Moebius strips are also obtained in this way.
	
	We will first recall a Weierstrass type formula for  isotropic (Willmore) surfaces in $\SSS^{4}$ \cite{Wang-S4}. Then we will discuss in details of two kind of one-parameter group action on isotropic surfaces in $\SSS^{4}$.  With help of the formula, we derive many explicit examples with expected properties in Section 5.3-5.6.
	
	We refer to \cite{Ejiri1988}, \cite{Mon}, \cite{Mus1}, \cite{Wang-S4} for more discussions of isotropic Willmore surfaces.
	\subsection{The Weierstrass formula for isotropic surfaces in $\SSS^{4}$}
	The following formula provides all explicit examples in this paper. So we include it here for readers' convenience.
	\begin{theorem}\label{th-willmore-iso-formula}\cite{Wang-S4}
		Let $M$ be a Riemann surface, and let
		\begin{equation}\label{eq-b1}\eta=\lambda^{-1}\left(
		\begin{array}{cc}
		0 & \hat{B}_1\\
		-\hat{B_1}I_{1,3} & 0 \\
		\end{array}
		\right)\dd z, \hbox { with }\hat{B}_1=\left(\mathbf{h}\ \ i\mathbf{h}\right)=\frac{1}{2}\left(
		\begin{array}{cccc}
		i(h_3'-h_2')&  -(h_3'-h_2') \\
		i(h_3'+h_2')&  -(h_3'+h_2')   \\
		h_4'-h_1' & i(h_4'-h_1')   \\
		i(h_4'+h_1') & -(h_4'+h_1')  \\
		\end{array}
		\right).\end{equation}
		Here $h_j$ are  meromorphic functions on $M$ satisfying $h_1'h_4'+h_2'h_3'=0, \hbox { and } h_1'h_2'\not\equiv0.$
		Then the corresponding Willmore surface $[Y_{\lambda}]$
		is of the form $Y_{\lambda} =R_{\lambda}Y_1$, with
		\begin{equation}\label{eq-willmore in s4-y-1}
		\begin{split}
		Y_1=\left(
		\begin{array}{c}
		y_0 \\
		y_1 \\
		y_2 \\
		y_3 \\
		y_4 \\
		y_5 \\
		\end{array}
		\right)
		=~&|h_1'|^2\left(\begin{array}{cc}
		(1+|h_2|^2+|h_4|^2)\\
		1-|h_2|^2+|h_4|^2\\
		-i (-\bar{h}_2h_4+h_2\bar{h}_4)\\
		-(\bar{h}_2h_4+h_2\bar{h}_4) \\
		i (\bar{h}_2-h_2) \\
		(\bar{h}_2+h_2) \\
		\end{array}\right)+|h_2'|^2\left(\begin{array}{cc}
		(1+|h_1|^2+|h_3|^2)\\
		-(1+|h_1|^2-|h_3|^2)\\
		i (-\bar{h}_1h_3+h_1\bar{h}_3)\\
		\bar{h}_1h_3+h_1\bar{h}_3 \\
		i (h_3-\bar{h}_3) \\
		-(h_3+\bar{h}_3) \\
		\end{array}\right)\\
		&~~~~+h_1'\bar{h}_2'\left(\begin{array}{cc}
		-\bar{h}_1h_2+\bar{h}_3h_4\\
		\bar{h}_1h_2+\bar{h}_3h_4\\
		-i (1+\bar{h}_1h_4+h_2\bar{h}_3)\\
		-(1-\bar{h}_1h_4+h_2\bar{h}_3)\\
		i (-\bar{h}_1+h_4) \\
		-(\bar{h}_1+h_4) \\
		\end{array}\right)+ \bar{h}_1'h_2'
		\overline{\left(\begin{array}{cc}
			-\bar{h}_1h_2+\bar{h}_3h_4\\
			\bar{h}_1h_2+\bar{h}_3h_4\\
			-i (1+\bar{h}_1h_4+h_2\bar{h}_3)\\
			-(1-\bar{h}_1h_4+h_2\bar{h}_3)\\
			i (-\bar{h}_1+h_4) \\
			-(\bar{h}_1+h_4) \\
			\end{array}\right)},\\
		\end{split}
		\end{equation}
		and
		\begin{equation}
		R_{\lambda}=\left(
		\begin{array}{ccccccc}
		1 & 0& 0 & 0 & 0 & 0 \\
		0 & 1& 0 & 0 & 0 & 0 \\
		0 & 0& 1 & 0& 0 & 0 \\
		0 & 0& 0 & 1& 0 & 0  \\
		0 & 0 & 0 & 0& \frac{\lambda+\lambda^{-1}}{2} &  \frac{\lambda-\lambda^{-1}}{-2i} \\
		0 & 0& 0 & 0&  \frac{\lambda-\lambda^{-1}}{2i} &  \frac{\lambda+\lambda^{-1}}{2}  \\
		\end{array}
		\right).
		\end{equation}
		$[Y_{\lambda}]$ is an (possibly branched) isotropic Willmore surface in $\SSS^{4}$.
		
		Moreover, a lift $\hat Y_1$ of the dual surface of $y_1=[Y_1]$ is of the form
		\begin{equation}\label{eq-willmore in s4-y-2}
		\begin{split}
		\hat Y_1=\left(
		\begin{array}{c}
		\hat{ y}_0 \\
		\hat{ y}_1 \\
		\hat{y}_2 \\
		\hat{y}_3 \\
		\hat{y}_4 \\
		\hat{ y}_5 \\
		\end{array}
		\right)=&|h_1'|^2\left(\begin{array}{cc}
		(1+|h_3|^2+|h_4|^2)\\
		-(1-|h_3|^2+|h_4|^2)\\
		-i (\bar{h}_3h_4-h_3\bar{h}_4)\\
		\bar{h}_3h_4+h_3\bar{h}_4\\
		i (-\bar{h}_3+h_3) \\
		-(\bar{h}_3+h_3) \\
		\end{array}\right)+
		|h_3'|^2\left(\begin{array}{cc}
		(1+|h_1|^2+|h_2|^2)\\
		1+|h_1|^2-|h_2|^2\\
		-i (-\bar{h}_1h_2+h_1\bar{h}_2)\\
		-\bar{h}_1h_2-h_1\bar{h}_2 \\
		-i (h_2-\bar{h}_2) \\
		h_2+\bar{h}_2 \\
		\end{array}\right)\\
		&~~~~+h_1'\bar{h}_3'\left(\begin{array}{cc}
		-\bar{h}_1h_3+\bar{h}_2h_4\\
		-\bar{h}_1h_3-\bar{h}_2h_4\\
		i(1+\bar{h}_1h_4+\bar{h}_2h_3)\\
		(1-\bar{h}_1h_4+\bar{h}_2h_3))\\
		i (\bar{h}_1-h_4) \\
		\bar{h}_1+h_4 \\
		\end{array}\right)+ \bar{h}_1'h_3'
		\overline{\left(\begin{array}{cc}
			-\bar{h}_1h_3+\bar{h}_2h_4\\
			-\bar{h}_1h_3-\bar{h}_2h_4\\
			i(1+\bar{h}_1h_4+\bar{h}_2h_3)\\
			(1-\bar{h}_1h_4+\bar{h}_2h_3))\\
			i (\bar{h}_1-h_4) \\
			\bar{h}_1+h_4 \\
			\end{array}\right)}.\\
		\end{split}
		\end{equation}
		and $\hat Y_{\lambda}= R_{\lambda} \hat Y_1$.
		
		Moreover we have
		\begin{enumerate}
			\item  $\hat Y_{\lambda}$ reduces to a point and $[Y_{\lambda}]$ is conformally equivalent to an isotropic minimal surface in $\R^4$, if and only if $h_3'=h_4'=0$;
			\item  Both $[Y_{\lambda}]$ and   $[\hat  Y_{\lambda}]$  are conformally equivalent to  (full) isotropic minimal surfaces in $\SSS^{4}$, if and only if there exists  a non-zero, real, constant vector $\mathbf{v}=(\mathrm{v}_1,\mathrm{v}_2,\mathrm{v}_3,\mathrm{v}_4)^t\in \R^4_1$ with $\mathbf{v}^tI_{1,3}\mathbf{v}=-1$, such that
			\begin{equation}\label{eq-w-iso-m-s}
			(-\mathrm{v}_3+i\mathrm{v}_4)h_1'+(\mathrm{v}_1+i\mathrm{v}_2)h_2'+(-\mathrm{v}_1+i\mathrm{v}_2)h_3'+( \mathrm{v}_3+i\mathrm{v}_4)h_4' =0;
			\end{equation}

			\item Both $[Y_{\lambda}]$ and  $[\hat  Y_{\lambda}]$  are conformally equivalent to   (full) isotropic minimal surfaces in $\mathbb H^4$, if and only if there exists  a non-zero, real, constant vector $\mathbf{v}=(\mathrm{v}_1,\mathrm{v}_2,\mathrm{v}_3,\mathrm{v}_4)^t\in \R^4_1$ with $\mathbf{v}^tI_{1,3}\mathbf{v}=1$, such that
			\begin{equation}\label{eq-w-iso-m-h}(-\mathrm{v}_3+i\mathrm{v}_4)h_1'+(\mathrm{v}_1+i\mathrm{v}_2)h_2'+(-\mathrm{v}_1+i\mathrm{v}_2)h_3'+( \mathrm{v}_3+i\mathrm{v}_4)h_4' =0.
			\end{equation}
		\end{enumerate}
	\end{theorem}

	\subsection{$S^1-$dressing action on isotropic surfaces in $\SSS^{4}$}
	Recall that for a isotropic surface in $\SSS^{4}$, its normalized potential has the form \cite{Wang-S4}
	\begin{equation}\label{eq-ve}
	\hat B_1=\left(\mathbf{h}\ \  i\mathbf{h}\right), \hbox{ with } \mathbf h=\frac{1}{2}\left(
	\begin{array}{cccc}
	i(h_3'-h_2')  \\
	i(h_3'+h_2')  \\
	h_4'-h_1'     \\
	i(h_4'+h_1')     \\
	\end{array}
	\right)\end{equation}
	Then we have
	\[T_{1,t}\mathbf h=\widetilde{\mathbf h}=\left(
	\begin{array}{cccc}
	i(\widetilde{h_3}'-\widetilde{h_2}')\\
	i(\widetilde{h_3}'+\widetilde{h_2}')\\
	h_4'-h_1' \\
	i(h_4'+h_1')    \\
	\end{array}
	\right), \hbox{ with }
	\widetilde{h_2}=e^{-it}h_2,\ \widetilde{h_3}=e^{it}h_3.
	\]
	By Theorem \ref{th-min-sn} and \ref{th-min-hn}, when $h_2=h_3$, we obtain a minimal surface in $\SSS^{4}$. When $  \widetilde{h_2}=  -\widetilde{h_3}$  we obtain a minimal surface in $\hh ^4$.
	\begin{proposition}\label{prop-s41} We retain the notions in Theorem \ref{th-willmore-iso-formula}. Assume now furthermore that $h_2=h_3$ in \eqref{eq-ve} and set $y_t=T_t\sharp y$, with $T_t=diag(T_{1,t}, I_2)$. Then
		\begin{enumerate}
			\item $y_t$ is conformally equivalent to a minimal surface in $\SSS^{4}$ if and only if $t=0$ or $\pi$;
			\item $y_t$ is conformally equivalent to a minimal surface in $\hh ^4$ if and only if $t=\frac{\pi}{2}$ or $\frac{3\pi}{2}$;
			\item $y_t$ is not conformally equivalent to any minimal surface in any space form for any $t\in (0,2\pi)$ and $t\not\in\{\frac{\pi}{2},\pi,\frac{3\pi}{2}\}$.
		\end{enumerate}
	\end{proposition}

	\begin{proposition} We retain the notions in Theorem \ref{th-willmore-iso-formula} and set  $T_t=diag(T_{2,t}, I_2)$.
		\begin{enumerate}
			\item  Assume  that $h_2=h_3$ and $y_t=T_t\sharp y$. Then $y_t$ is conformally equivalent to a minimal surface in $\SSS^{4}$ for any $t\in\R$.
			\item  Assume  that $h_2=-h_3$ and $y_t=T_t\sharp y$. Then $y_t$ is conformally equivalent to a minimal surface in $\hh^4$ for any $t\in\R$.
		\end{enumerate}
	\end{proposition}
	
	\subsection{The Veronese sphere and its  Willmore deformations}

	Applying to the Veronese surface in $\SSS^{4}$, we obtain many new Willmore two-spheres in $\SSS^{4}$ with the same Willmore energy. Moreover, we also obtain many examples of minimal surfaces in $\hh^4$ with Willmore energy taking every value in $(0,2\pi)$.
	\subsubsection{The Veronese sphere and its $S^1-$Willmore deformations}
	\begin{proposition}\label{prop-h41}
		Let $z=re^{i\theta}$. Set
		\begin{equation}\label{eq-f}
		h_1=-2z^3,\ h_2=\sqrt{3}iz^2,\ h_3=\sqrt{3}iz^2, \ h_4=-2z,
		\end{equation}
		in \eqref{eq-ve}. Let $[Y]$ be the corresponding Willmore surface in $S^4$.
		Set $Y_t=T_t\sharp Y$ with $T_t=diag(T_{1,t}, I_2)$. Then
		\begin{equation}\label{eq-ve-y}
		Y_t=\left(
		\begin{array}{ccccc}
		r^4+2r^2+1\\
		-r^4+4r^2-1\\
		\frac{\sqrt{3}\left({z}{e^{-it}}+\bar ze^{it}-r^4\left(ze^{it}+{\bar z}{e^{-it}}\right)\right)}{1+r^2}\\
		\frac{-i\sqrt{3}\left({z}{e^{-it}}-\bar ze^{it}-r^4\left(ze^{it}-{\bar z}{e^{-it}}\right)\right)}{1+r^2}\\
		\frac{\sqrt{3}\left({z^2}{e^{-it}}+\bar z^2e^{it}+r^2\left(z^2e^{it}+{\bar z^2}{e^{-it}}\right)\right)}{1+r^2} \\
		\frac{i\sqrt{3}\left({z^2}{e^{-it}}-\bar z^2e^{it}+r^2\left(z^2e^{it}-{\bar z^2}{e^{-it}}\right)\right)}{1+r^2} \\
		\end{array}
		\right)=\left(
		\begin{array}{ccccc}
		r^4+2r^2+1\\
		-r^4+4r^2-1\\
		\frac{2\sqrt{3}r\left(\cos(\theta-t)-r^4\cos(\theta+t)\right)}{1+r^2}\\
		\frac{2\sqrt{3}r\left(\sin(\theta-t)-r^4\sin(\theta+t)\right)}{1+r^2}\\
		\frac{2\sqrt{3}r^2\left(\cos(2\theta-t)+r^2\cos(2\theta+t)\right)}{1+r^2}\\
		\frac{-2\sqrt{3}r^2\left(\sin(2\theta-t)+r^2\sin(2\theta+t)\right)}{1+r^2}\\
		\end{array}
		\right).
		\end{equation}
		and $y_t=[Y_t]:S^2\rightarrow S^4$ is an isotropic Willmore immersion with
		\begin{equation}\label{eq-ve-yt}
		y_t= \frac{1}{(r^2+1)^3}\left(
		\begin{array}{ccccc}
		-r^6+3r^4+3r^2-1\\
		2\sqrt{3}r\left(\cos(\theta-t)-r^4\cos(\theta+t)\right)\\
		{2\sqrt{3}r\left(\sin(\theta-t)-r^4\sin(\theta+t)\right)} \\
		2\sqrt{3}r^2\left(\cos(2\theta-t)+r^2\cos(2\theta+t)\right) \\
		-2\sqrt{3}r^2\left(\sin(2\theta-t)+r^2\sin(2\theta+t)\right)\\
		\end{array}
		\right)
		\end{equation}
		and
		\begin{equation}\label{eq-yt-metric}
		|\dd y_t|^2= \frac{12(r^8 + 4r^6 +6r^4\cos 2t + 4r^2 + 1)}{(r^2+1)^6}|\dd z|^2.
		\end{equation}
		\begin{enumerate}
			\item $W([Y_t])=8\pi$ for all $t\in[0,2\pi]$. $[Y_t]$ is conformally equivalent to $[Y_{t+\pi}]$ for all $t\in[0,\pi]$. And for any $t_1,t_2\in[0,\pi)$, $[Y_{t_1}]$ is conformally equivalent to $[Y_{t_2}]$ if and only if $t_1=t_2$ or $t_1+t_2=\pi$.
			\item $[Y_t]$ is conformally equivalent to the Veronese surface in $\SSS^{4}$ when $t=0$ and  $[Y_t]$ is conformally equivalent to three complete minimal surfaces in $\hh^4$ on three open subsets of $S^2$ when $t=\frac{\pi}{2}$. For any  other $t\in(0,\pi)$, $[Y_t]$ is a Willmore surface in $\SSS^{4}$ not minimal in any space form.
			\item When $t=\frac{3\pi}{2}$, consider the projection of $[(Y_t)|_{t=\frac{3\pi}{2}}]$ into $\hh^4$ w.r.t $(0,1,0,0,0,0)^t\in\R^6_1$:
			\begin{equation}\label{eq-ve-y-hmin}
			\widetilde{y}=\frac{-1}{(1+r^2)(r^4-4r^2+1)}\left(
			\begin{array}{cccccc}
			(1+r^2)^3\\
			\sqrt{3}i(z-\bar{z} ) (1+r^4 )  \\
			\sqrt{3}(z+\bar{z} ) (1+r^4 )  \\
			\sqrt{3}i(z^2-\bar{z}^2 )(1-r^2 )  \\
			{-\sqrt{3}(z^2+\bar{z}^2 )(1-r^2 )} \\
			\end{array}
			\right).
			\end{equation}
			It has metric
			\[|\dd\widetilde{y}|^2=\frac{12(r^8+4r^6-6r^4+4r^2+1)}{(r^2+1)^2(r^4-4r^2+1)^2}|\dd z|^2\]
			and Gauss curvature
			\begin{equation}\label{eq-ve-y-hmin-K}
			\begin{split}
			K=&-1-\frac{2}{3}\frac{(r^2+1)^4(r^4-4r^2+1)^4}{(r^8+4r^6-6r^4+4r^2+1)^3}
			\end{split}\end{equation}
			on $S^2\setminus \{|z|=r_1\}\cup\{|z|=r_2\}$. Here $r_1=\frac{\sqrt{6}-\sqrt{2}}{2}$ and $r_2=\frac{\sqrt{6}+\sqrt{2}}{2}$. Set
			\[
			M_1=\{z\in \overline{\mathbb C}\ |\ |z| <r_1\},\ M_2=\{z\in \overline{\mathbb C}\ |\ r_1<|z| <r_2\},\ M_3=\{z\in \overline{\mathbb C}\ |\ |z|>r_2\}.
			\]
			
			\begin{enumerate}
				\item Set $\mu(z):=-\frac{1}{\bar{z}}$ on $S^2$. Then
				\[\widetilde{y}\circ \mu=R\widetilde{y},~ \hbox{  with } R=\hbox{diag}(1,-1,-1,1,1).
				\]

				\item   $\widetilde y|_{M_1}:M_1\rightarrow \hh^4$ is a proper, complete minimal disk with finite Willmore energy $(4-2\sqrt{3})\pi$. Its Gauss curvature takes value in $[-\frac{5}{3},-1)$. In particular, it has bounded Gauss curvature.  And $\widetilde y|_{M_3}$ is congruent to $\widetilde y|_{M_1}$ in the sense  $\widetilde y|_{M_3}=R(\widetilde y\circ\sigma)|_{M_1}$.
				\item  $\widetilde y|_{M_2}:M_2\rightarrow \hh^4$ is a proper, complete minimal annulus with  finite Willmore energy $4\sqrt 3 \pi$. Its Gauss curvature takes value in $[-\frac{11}{3},-1)$. In particular, it has bounded Gauss curvature.
				\item Each of the three minimal surfaces intersects the infinite boundary of $\hh^4$  with a constant angle $<\frac{\pi}{2}$. The circles $r=\frac{\sqrt{6}\pm\sqrt{2}}{2}$ are the umbilical sets of the Willmore immersion $[Y_{t}|_{t=\frac{3\pi}{2}}]$.
			\end{enumerate}\end{enumerate}
		\end{proposition}
		\begin{proof} The equation \eqref{eq-ve-y} is a direct application of Theorem \ref{th-min-sn} and Theorem \ref{th-willmore-iso-formula}.
			When $t=0$, we see that $y_t|_{t=0}$ is a minimal immersion with constant curvature $1/3$, hence it is the Veronese surface. It is well-known that Veronese two-sphere have Willmore energy $8\pi$. Since the Willmore energy of $[Y_t]$ depends smoothly on $t$ and the Willmore energy of a Willmore two-sphere is $4\pi m$ for some $m\in Z$ \cite{Mon}, we see that $W([Y_t])=8\pi$.
			
			Substituting  $t+\pi$ into  \eqref{eq-ve-y} we see that $[Y_t]$ is conformally equivalent to $[Y_{t+\pi}]$. By  Theorem \ref{th-min-sn}, we see that (2) holds. From \eqref{eq-ve-y} we see that for any $[Y_t]$, it admits an $S^1-$symmetry given by $R_{\tilde{t}}=\hbox{diag}(I_2,R_{\tilde t}, R_{2\tilde t})$. Here
			\[R_{\tilde t}=\left(
			\begin{array}{cc}
			\cos \tilde t & -\sin \tilde t \\
			\sin \tilde t& \cos \tilde t \\
			\end{array}
			\right),\ R_{2\tilde t}=\left(
			\begin{array}{cc}
			\cos 2\tilde t &  \sin 2\tilde t \\
			- \sin 2\tilde t& \cos 2\tilde t \\
			\end{array}
			\right).
			\]
			To be concrete, we have $Y_t(ze^{i\tilde t}, \bar z e^{-i\tilde t})=R_{\tilde{t}} Y_t$.
			Moreover, for any $t\in(0,\pi)$, $[Y_t]$ does not admit another $S^1-$symmetry. Otherwise, we will see that $[Y_t]$ is a homogeneous Willmore two-sphere since it has two different $S^1-$symmetry. By \cite{MPW,DW-h},  it is conformally equivalent to the Veronese two-sphere, which is not possible. Therefore, $[Y_{t_1}]$ is conformally equivalent to $[Y_{t_2}]$ only if $y_{t_1}$ is isometric to $y_{t_2}$, which by \eqref{eq-yt-metric}, if and only if $t_1=t_2$ or $t_1+t_2=\pi$. By \eqref{eq-ve-y}, $[Y_{t_1}]$ is conformally equivalent to $[Y_{t_2}]$ if $t_1+t_2=\pi$. This finishes (1).
			And (2) comes from Theorem \ref{th-min-sn}.

			(3) comes from a lengthy  but straightforward computation. Note that the properness of $\widetilde{y}|_{M_j}$, $j=1,2,3$ comes from the fact that they have smooth boundary curves at infinity.
		\end{proof}
		\begin{remark}\
			\begin{enumerate}
				\item  Note that K attains maximal value $-1$ at $r=\frac{\sqrt{6}\pm\sqrt{2}}{2}$ and attains  minimal value $-\frac{11}{3}$ at $r=1$ (FIGURE 1). This means that the two circles $r=\frac{\sqrt{6}\pm\sqrt{2}}{2}$ on $S^2=\bar\C$ are exactly the umbilical sets on the Willmore surface $[Y_t]|_{t=\frac{3\pi}{2}}$ (Compare also \cite{Ba-Bo}).
				\begin{figure}[h]
					\centering
					\includegraphics[width=0.40\textwidth]{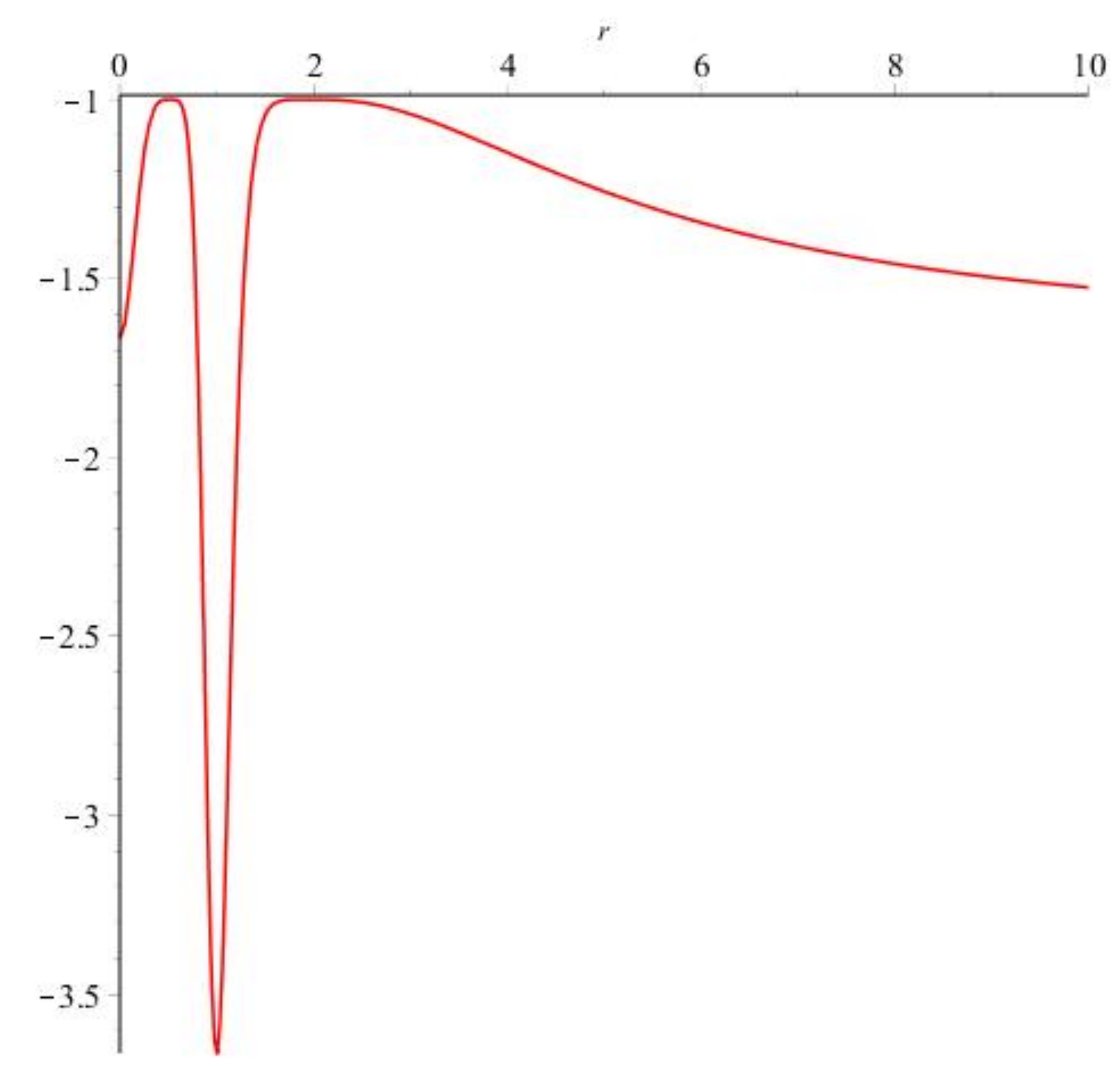}
					\caption{Curvature of $\widetilde{y}$}\label{fig:digit}
				\end{figure}
				\item
				The surface $[\widetilde{Y}_{t=\frac{3\pi}{2}}]$ can be looked as a combination of three complete minimal surfaces $\widetilde{y}|_{M_j}$ in $\mathbb H^4$, with $j=1,2,3$. To be concrete, when $|z|<\frac{\sqrt{6}-\sqrt{2}}{2}$, the surface takes values in the upper connected component of $\mathbb H^4$, and tends to the boundary of $\mathbb H^4$ when $|z|\rightarrow\frac{\sqrt{6}-\sqrt{2}}{2}$ from the left side. When $|z|=\frac{\sqrt{6}-\sqrt{2}}{2}$, it  takes values at the boundary of $\mathbb H^4$. When   $\frac{\sqrt{6}-\sqrt{2}}{2}<|z|<\frac{\sqrt{6}+\sqrt{2}}{2}$, it takes values in the lower connected component of $\mathbb H^4$, and tends to the boundary of $\mathbb H^4$ again when $|z|\rightarrow\frac{\sqrt{6}+\sqrt{2}}{2}$ from the right side. When $|z|=\frac{\sqrt{6}+\sqrt{2}}{2}$, it  takes values at the boundary of $\mathbb H^4$ again. When $|z|>\frac{\sqrt{6}+\sqrt{2}}{2}$, it again takes values in the upper connected component of $\mathbb H^4$. When viewing the surface in $\mathbb H^4$, it blows up at the points $\frac{\sqrt{6}\pm\sqrt{2}}{2}$. If we embed $\mathbb H^4$  conformally into $\SSS^{4}$, the surface will be a smooth immersion on the whole $S^2$. This is the well-known construction of compact Willmore surfaces due to Babich and Bobenko \cite{Ba-Bo} for minimal surfaces in $\mathbb H^3$, where they constructed successfully Willmore tori with a umbilical line in $S^3$ via this way.  It is hence not surprising that similar construction also works for Willmore two-spheres. To the authors' best knowledge, the example in Proposition \ref{prop-h41} should be the first explicit example of Willmore two-sphere in $\SSS^{4}$ which is conformally equivalent to some minimal surface in $\mathbb H^4$ on an open subset of $S^2$ (Note that this is not possible for Willmore two-spheres in $S^3$ except the round sphere \cite{Bryant1984}).
				
				\item In \cite{Mon}, it is shown that all Willmore two-spheres with $W([Y_t])=8\pi$ are expressed as twistor deformations of the Veronese surface in $\SSS^{4}$. Here we derive some explicit examples. Moreover, the generating curve of the $S^1-$equivariant Willmore two-sphere  $y_t$, $t\in(0,\frac{\pi}{2})$, in $S^4$ is
				\[\gamma_t=\frac{1}{(r^2+1)^3}\left(
				\begin{array}{ccccc}
				-r^6+3r^4+3r^2-1\\
				2\sqrt{3}r\left(1-r^4 \right)\cos t\\
				- 2\sqrt{3}r\left(1+r^4\right)\sin t \\
				2\sqrt{3}r^2\left(1 +r^2\right)\cos t \\
				2\sqrt{3}r^2\left(1-r^2\right)\sin t\\
				\end{array}
				\right).\]
				So $\gamma_t$ is full in $S^4$ for all $t\in(0,\frac{\pi}{2})$ and  $\gamma_t$ takes value in some $S^2\subset S^4$ when $t=0,\frac{\pi}{2}$. This indicates that in general, $S^1-$equivariant Willmore two-spheres in $S^4$ have more complicated structures than $S^1-$equivariant minimal two-spheres in $S^4$ \cite{Ejiri-equ}.
				
				\item Different from the case of complete minimal surfaces in $\hh^3$ with finite Willmore energy, which always intersect the infinity boundary orthogonally as shown in \cite{AM2015}, here the complete minimal surface $\widetilde{y}$ intersect the infinity boundary with a constant angle not equal to $\frac{\pi}{2}$.
			\end{enumerate} \end{remark}

			\subsubsection{$\R^1-$minimal deformations of the minimal surface $\widetilde y$}
			Let us consider the $\R^1-$minimal deformations of the minimal surface $\widetilde y$  in $\hh^4$ given in  \eqref{eq-ve-y-hmin}, by use of which we obtain a lot of (non-congruent) complete minimal surfaces in $\hh^4$.
			\begin{proposition}\label{prop-h42}
				Let $z=re^{i\theta}$. Set
				\begin{equation}\label{eq-f}
				h_1=-2z^3,\ h_2=\sqrt{3}z^2,\ h_3=-\sqrt{3}z^2, \ h_4=-2z.
				\end{equation}
				in \eqref{eq-ve}. Let $[Y]$ be the corresponding Willmore surface in $S^4$.
				Set $Y_t=T_t\sharp Y$ with $T_t=diag(T_{2,t}, I_2)$. Then
				\begin{equation}\label{eq-ve-y-5}
				Y_t=\left(
				\begin{array}{c}
				y_0 \\
				y_1 \\
				y_2 \\
				y_3 \\
				y_4 \\
				y_5 \\
				\end{array}
				\right)
				=\left(
				\begin{array}{ccccc}
				e^{2t}r^6+3r^4+3e^{2t}r^2+1\\
				-e^{2t}r^6+3r^4+3e^{2t}r^2-1\\
				i\sqrt{3}e^t(1+r^4)(z-\bar z)\\
				\sqrt{3}e^t(1+r^4)(z+\bar z)\\
				i\sqrt{3} (1-e^{2t}r^2)(z^2-\bar z^2)\\
				-\sqrt{3} (1-e^{2t}r^2)(z^2+\bar z^2)\\
				\end{array}
				\right).
				\end{equation}
				\begin{enumerate}
					\item For every $t\in\R$, $[Y_t]$ is a Willmore immersion from $S^2$ to $\SSS^{4}$  with $W([Y_t])=8\pi$ and  $[Y_t]$ is oriented for all $t\in\R$.
					Moreover, $[Y_t(z,\bar z)]$ is conformally equivalent to $[Y_{-t}(-\frac{1}{z},-\frac{1}{\bar z})]$ for all $t\in\R.$
					
					\item Set
					\[y_t=\frac{1}{y_1}\left(
					\begin{array}{ccccc}
					y_0 & y_2 & y_3 & y_4 & y_5 \\
					\end{array}
					\right)^t.\]
					Then $y_t$ is minimally immersed  into $\hh^4$ on the points where $y_0\neq0$, with metric
					\[\begin{split}|(y_t)_z|^2&=\frac{6(e^{2t}r^8 + 4e^{4t}r^6 - 6e^{2t}r^4 + 4r^2+e^{2t})}{(e^{2t}r^6-3r^4-3e^{2t}r^2+1)^2}\\
					&=\frac{6\left(e^{2t}(r^4+1)^2 + 4r^2(e^{2t}r^2-1)^2\right)}{(e^{2t}r^6-3r^4-3e^{2t}r^2+1)^2}
					\end{split}
					\]
					and curvature
					\[K=-1-\frac{2e^{2t}(e^{2t}r^6- 3r^4 - 3e^{2t}r^2 + 1)^4}{3(e^{2t}r^8 + 4e^{4t}r^6 - 6e^{2t}r^4 + 4r^2+e^{2t})^3}\]
					In particular, set
					\[\begin{split}
					&M_{t,1}=\{z\in \overline{\mathbb C}\ |\ |z| <r_1\},\\
					&M_{t,2}=\{z\in \overline{\mathbb C}\ |\ r_1<|z| <r_2\},\\
					&M_{t,3}=\{z\in \overline{\mathbb C}\ |\ |z|>r_2\}.
					\end{split}\]
					Here we denote by $r_1$ and $r_2$ the two positive solutions to $e^{2t}r^6-3r^4-3e^{2t}r^2+1=0$ with $0<r_1<r_2$
					\footnote{Note that $\cos3\theta_0-2\cos (\theta_0+\frac{\pi}{3})=2\sin\theta_0(\sin\frac{\pi}{3}-\sin2\theta_0)>0$ since $0<\theta_0< \pi/6$ for all $t\in\R$.}:
					\[r_1^2=\sqrt{1+e^{-4t}}\left(\cos3\theta_0-2\cos (\theta_0+\frac{\pi}{3})\right),\ r_2^2=\sqrt{1+e^{-4t}}\left(\cos3\theta_0+2\cos \theta_0\right).\]
					Here $\theta_0=\frac{1}{3}\arccos\frac{1}{\sqrt{1+e^{4t}}}.$
					Then we obtain two complete minimal disks $M_{t,1}$, $M_{t,3}$ and one complete minimal annulus $M_{t,2}$ in $\hh^4$.
					\item $[Y_t]|_{M_{t,1}}$ and $[Y_t]|_{M_{t,3}}$ are conformally equivalent to complete immersed, isotropic minimal disks $y_{t,1}$ and $y_{t,3}$ in $\hh^4$.
					Moreover,  $y_{t,1}$ and $y_{t,3}$ are isometrically congruent if and only if $t=0$.  $[Y_t]|_{M_{t,2}}$ is conformally equivalent to an immersed, complete, isotropic minimal annulus $y_{t,2}$ in $\hh^4$.
					\item
					When $t\rightarrow +\infty$, $[Y_t]$ tends to
					a branched double cover of a totally geodesic surface $y_{\infty}$ $\SSS^4$ which is orthogonal to the equator $\SSS^3_0=\{x\in \SSS^4|x\perp (1,0,0,0,0)^t\}$.
					\item When $t=0$, $W([Y_t]|_{M_{t,1}})=W([Y_t]|_{M_{t,3}})=(4-2\sqrt{3})\pi$. When $t\rightarrow +\infty$, $W([Y_t]|_{M_{t,1}})\rightarrow 0$. There exists $t\in\R^-$ such that $W([Y_t]|_{M_{t,1}})>1.9999\pi$.
					Hence for every $c_0\in(0,1.9999\pi]$, there exists some $t\in\R$ such that $W([Y_t]|_{M_{t,1}})=c_0$.
				\end{enumerate}   \end{proposition}

				\begin{proof}
					(1) and (2) come from direct computations, as shown in the proposition. (3) is obvious.
					
					Now let's consider (4).  When $t\rightarrow +\infty$, from \eqref{eq-ve-y-5} it is direct to see that $y_t$ tends to
					\[\frac{1}{3+r^4}\left(
					\begin{array}{ccccc}
					-r^4+3\\
					0\\
					0\\
					-i\sqrt{3}(z^2-\bar z^2)\\
					\sqrt{3}(z^2+\bar z^2)\\
					\end{array}
					\right),\]
					which is exactly a branched double covering of a totally geodesic surface $y_\infty$ orthogonal to the infinity boundary of $\hh^4$.
					Moreover, $[Y_t]|_{M_{t,1}}$ tends to the branched point $p_0=(1,0,0,0,0)^t$.
					The equator $\SSS^3_0$ divides $y_\infty$ into two parts: $y_{\infty}^+$ (containing $p_0$) and $y_{\infty}^-$.
					Therefore $[Y_t]|_{M_{t,2}}$ tends to $y_{\infty}^+\setminus\{p_0\}$ and $[Y_t]|_{M_{t,3}}$ tends to $y_{\infty}^-$ .
					
					Finally, let's consider (5). First we note that the Willmore energy of $[Y_t]|_{M_{t,j}}$ are
					\begin{equation}\label{eq-yt-W}
					W(M_{t,j})=16\pi\int_{r_{j-1}}^{r_j}\frac{ e^{2t}(e^{2t}r^6- 3r^4 - 3e^{2t}r^2 + 1)^2}{(e^{2t}r^8 + 4e^{4t}r^6 - 6e^{2t}r^4 + 4r^2+e^{2t})^2}r \dd r,~ j=1,2,3,
					\end{equation}
					with $r_0=0$, $r_3=+\infty$ and $r_1$ and $r_2$ as shown in the proposition.
					
					Since \[\lim_{t\rightarrow+\infty}r_1=0 \hbox{ and } \lim_{t\rightarrow+\infty}e^{2t}r_1=0,\]
					when $t\rightarrow+\infty$ we have for $0\leq r\leq r_1$
					\[(e^{2t}r^8 + 4e^{4t}r^6 - 6e^{2t}r^4 + 4r^2+e^{2t})^2\geq e^{4t},\ (e^{2t}r^6- 3r^4 - 3e^{2t}r^2 + 1)^2<1.\]
					So when $t\rightarrow+\infty$,
					\[\begin{split}
					\int_{0}^{r_1}\frac{ e^{2t}(e^{2t}r^6- 3r^4 - 3e^{2t}r^2 + 1)^2}{(e^{2t}r^8 + 4e^{4t}r^6 - 6e^{2t}r^4 + 4r^2+e^{2t})^2}r \dd r& \leq \int_{0}^{r_1} e^{-2t}r \dd r\\
					&=2 e^{-2t}r_1^2.\\
					\end{split}\]
					So \[\lim_{t\rightarrow+\infty}W(M_{t,1})=0.\]
					
					On the other hand, numerical computation shows when $t=\ln0.000039$,
					\[W(M_{t,2})\approx 6.000089931\pi, \ W(M_{t,1})\approx 1.999910062\pi.\]
					Since $W(M_1)$ depends continuously on $t$, we see that for any number $c_0\in(0,1.9999\pi]$, there exists some $t_0\in\R$ such that
					$W(M_1)=c_0$ for $t=t_0$. This finishes the proof.
				\end{proof}
				
				\begin{remark} It is interesting to ask whether there exists a complete minimal annulus $x$ in $\hh^4$ with $W(x)\leq 6\pi$. Moreover, what is the infimum of the Willmore energy of a complete minimal annulus $x$ in $\hh^4$?
				\end{remark}

				\subsubsection{$\R^1-$minimal deformations of the Veronese two-sphere in $\SSS^{4}$}
				Similarly we can construct a family of minimal two-spheres in $\SSS^4$ via the $\R^1-$action on the Veronese two sphere in $\SSS^{4}$.
				\begin{proposition}
					Let $z=re^{i\theta}$. Set
					\begin{equation}\label{eq-f}
					h_1=-2z^3,\ h_2=\sqrt{3}iz^2,\ h_3=\sqrt{3}iz^2, \ h_4=-2z.
					\end{equation}
					Set $Y_t=T_t\sharp Y$ with $T_t=diag(T_{2,t}, I_2)$. Then
					\begin{equation}\label{eq-ve-y-t2}
					Y_t=\left(
					\begin{array}{c}
					y_0 \\
					y_1 \\
					y_2 \\
					y_3 \\
					y_4 \\
					y_5 \\
					\end{array}
					\right)
					=\left(
					\begin{array}{ccccc}
					e^{2t}r^6+3r^4+3e^{2t}r^2+1\\
					-e^{2t}r^6+3r^4+3e^{2t}r^2-1\\
					\sqrt{3}e^t(1-r^4)(z+\bar z)\\
					-i\sqrt{3}e^t(1-r^4)(z-\bar z)\\
					\sqrt{3} (1+e^{2t}r^2)(z^2+\bar z^2)\\
					i\sqrt{3} (1+e^{2t}r^2)(z^2-\bar z^2)\\
					\end{array}
					\right).
					\end{equation}
					\begin{enumerate}
						\item For every $t\in\R$, $[Y_t]$  is conformally equivalent to  an immersed isotropic minimal two-sphere $y_t=\frac{1}{y_0}\left(
						\begin{array}{ccccc}
						y_1 & y_2 & y_3 & y_4 & y_5 \\
						\end{array}
						\right)^t
						$ in $\SSS^{4}$ with $W([Y_t])=8\pi$,
						\[\left|\dd(y_t)\right|^2=\frac{12(e^{2t}r^8 + 4e^{4t}r^6 + 6e^{2t}r^4 + 4r^2 + e^{2t})}{(e^{2t}r^6+3r^4+3e^{2t}r^2+1)^2}|\dd z|^2\]
						and
						\begin{equation}\label{eq-min-sn-k}
						K_t=1-\frac{2e^{2t}(e^{2t}r^6+3r^4+3e^{2t}r^2+1)^4}{3(e^{2t}r^8 + 4e^{4t}r^6 + 6e^{2t}r^4 + 4r^2 + e^{2t})^3}.
						\end{equation}
						
						\item $[Y_t]$ descend to a minimal $\R P^2$ if and only if $t=0$.
						\item When $t\rightarrow\infty$, $y_t$ tends to a branched double covering of a totally geodesic round two-sphere of $\SSS^4$.
					\end{enumerate}   \end{proposition}

					\subsection{$S^1-$deformation of generalizations of Veronese two-sphere in $\SSS^4$}
					In \cite{Do-Wa-equ}, generalizations of Veronese two-sphere in $\SSS^4$ are discussed. Here we consider the $S^1-$deformation of them, which will give more examples of complete minimal surfaces in $\hh^4$, which will be important in Willmore energy estimates of complete minimal surfaces in $\hh^4$.
					\begin{proposition}\label{prop-s4k}
						Let $z=re^{i\theta}$. Set
						\begin{equation}\label{eq-f2k}
						h_1=-kz^{k+1},\ h_2=i\sqrt{k^2-1}z^{k},\ h_3=i\sqrt{k^2-1}z^{k}, \ h_4=-kz^{k-1},~ k\geq 2,
						\end{equation}
						in \eqref{eq-ve}. Let $[\hat Y]$ be the corresponding Willmore surface in $S^4$.
						Set $\hat Y_t=T_t\sharp \hat Y$ with $T_t=diag(T_{1,t}, I_2)$. Then
						\begin{equation}\label{eq-ve-y-k}
						\hat Y_t=\left(
						\begin{array}{c}
						\hat   y_0 \\
						\hat y_1 \\
						\hat y_2 \\
						\hat y_3 \\
						\hat  y_4 \\
						\hat y_5 \\
						\end{array}
						\right)
						=\left(
						\begin{array}{ccccc}
						(k-1)(r^{2k+2}+1)+(k+1)(r^{2k}+r^2)\\
						-(k-1)(r^{2k+2}+1)+(k+1)(r^{2k}+r^2)\\
						\sqrt{k^2-1}\left((ze^{-it}+\bar z e^{it})-r^{2k}(ze^{it}+\bar z e^{-it})\right)\\
						-i\sqrt{k^2-1}\left((ze^{-it}-\bar z e^{it})-r^{2k}(ze^{it}-\bar z e^{-it})\right)\\
						\sqrt{k^2-1}\left((z^{k}e^{-it}+\bar z^{k}e^{it})+r^{2}(z^{k}e^{it}+\bar z^{k}e^{-it})\right)\\
						i\sqrt{k^2-1}\left((z^{k}e^{-it}-\bar z^{k}e^{it})+r^{2}(z^{k}e^{it}-\bar z^{k}e^{-it})\right)\\
						\end{array}
						\right).
						\end{equation}
						\begin{enumerate}
							\item For every $t\in[0,2\pi]$, $[\hat Y_t]$ is an oriented Willmore immersion from $S^2$ to $\SSS^{4}$ with Willmore energy $4\pi k $.
							$[\hat Y_t]$ is conformally equivalent to $[\hat Y_{t+\pi}]$ for all $t\in[0,\pi]$. And for any $t_1,t_2\in[0,\pi)$, $[\hat Y_{t_1}]$ is conformally equivalent to $[\hat Y_{t_2}]$ if and only if $t_1=t_2$ or $t_1+t_2=\pi$.
							
							\item $[\hat Y_t]$ is conformally equivalent to a minimal two-sphere in $\SSS^{4}$ when $t=0$ and  $[Y_t]$ is conformally equivalent to three complete minimal surfaces in $\hh^4$ on three open subsets of $S^2$ when $t=\frac{\pi}{2}$. For any  other $t\in(0,\pi)$, $[Y_t]$  Willmore surfaces in $\SSS^{4}$ not minimal in any space form.
							\item $[\hat Y_t]$ reduces to a non-oriented Willmore surface from $\R P^2=S^2/\mu$, if and only if $t=0$ or $\pi$, and $k=2\tilde k$ for some $\tilde k\in \mathbb Z^+$.  Here $\mu(z)=-\frac{1}{\bar z}$.
						\end{enumerate}   \end{proposition}
						\begin{proof}
							The equation \eqref{eq-ve-y-k} comes from direct computations.
							
							We need only to show that $W([\hat Y_t])=4\pi k$, since proofs of the rest of (1) and (2) are the same as Proposition \ref{prop-h41}.
							Since the Willmore energy of $[Y_t]$ depends smoothly on $t$ and the Willmore energy of a Willmore two-sphere is $4\pi m$ for some $m\in Z$ \cite{Mon}, we have
							$Area([\hat Y_t])=Area([\hat Y])$. By Theorem 3.1 of \cite{Ejiri-equ} (see also \cite{Bab}), $Area([\hat Y_t])=Area([\hat Y])=4\pi(k+1)$ since the equivariant action here is $(m_{(1)},m_{(2)})=(1,k)$.
							
							Substituting $\mu$ into \eqref{eq-ve-y-k} shows that
							$[\hat Y_t\circ\mu]=[\hat Y_t]$ if and only if $k$ is even and $t=0$ or $\pi$, which  finishes the proof of (3).
						\end{proof}
						\subsection{$\R^1-$minimal deformations of another type of minimal surfaces in $\hh^4$}
						
						It is natural to show the existence of complete minimal surfaces in $\hh^4$ with any Willmore energy $W_0\in\R^+\cup\{0\}$ by further generalization of the above examples.
						
						\begin{proposition}\label{prop-h4k}
							Let $z=re^{i\theta}$. Let $[Y]=\hat Y_{t}|_{t=\frac{3\pi}{2}}$. Then its normalized potential can be given by setting
							\begin{equation}\label{eq-f2}
							h_1=-kz^{k+1},\ h_2=\sqrt{k^2-1}z^{k},\ h_3=-\sqrt{k^2-1}z^{k}, \ h_4=-kz^{k-1},~ k\geq 2,
							\end{equation}
							in \eqref{eq-ve}.  Set $Y_t=T_t\sharp Y$ with $T_t=diag(T_{2,t}, I_2)$. Then
							\begin{equation}\label{eq-ve-y-5}
							Y_t=\left(
							\begin{array}{c}
							y_0 \\
							y_1 \\
							y_2 \\
							y_3 \\
							y_4 \\
							y_5 \\
							\end{array}
							\right)
							=\left(
							\begin{array}{ccccc}
							(k-1)(e^{2t}r^{2k+2}+1)+(k+1)(r^{2k}+e^{2t}r^2)\\
							-(k-1)(e^{2t}r^{2k+2}+1)+(k+1)(r^{2k}+e^{2t}r^2)\\
							i e^t\sqrt{k^2-1}(1+r^{2k})(z-\bar z)\\
							e^t\sqrt{k^2-1}(1+r^{2k})(z+\bar z)\\
							i \sqrt{k^2-1}(1-e^{2t}r^{2})(z^{k}-\bar z^{k})\\
							-\sqrt{k^2-1}(1-e^{2t}r^{2})(z^{k}+\bar z^{k})\\
							\end{array}
							\right).
							\end{equation}
							\begin{enumerate}
								\item For every $t\in\R$, $[Y_t]$ is an oriented Willmore immersion from $S^2$ to $\SSS^{4}$ with Willmore energy $4\pi k $
								and
								$[Y_t(z,\bar z)]$ is conformally equivalent to $[Y_{-t}(-\frac{1}{z},-\frac{1}{\bar z})]$.
								
								\item Set
								\[y_t=\frac{1}{y_1}\left(
								\begin{array}{ccccc}
								y_0 & y_2 & y_3 & y_4 & y_5 \\
								\end{array}
								\right)^t.\]
								Then $y_t$ is minimally immersed  into $\hh^4$ on the points where $y_0\neq0$, with metric
								\[\begin{split}|\dd y_t|^2&=\frac{4(k^2-1)\left(e^{2t}(1+r^{2k})^2 + k^2r^{2k-2}(1-e^{2t}r^2)^2 \right)}{\left((k-1)(e^{2t}r^{2k+2}+1)-(k+1)(r^{2k}+e^{2t}r^2)\right)^2}|\dd z|^2
								\end{split}
								\]
								and curvature
								\[K=-1-\frac{k^2 e^{2t} r^{2k-4}\left((k-1)(e^{2t}r^{2k+2}+1)-(k+1)(r^{2k}+e^{2t}r^2)\right)^4}{2(k^2-1)\left(e^{2t}(1+r^{2k})^2 + k^2r^{2k-2}(1-e^{2t}r^2)^2 \right)^3}\]
								In particular, set
								\[\begin{split}
								&M_{t,1}=\{z\in \overline{\mathbb C}\ |\ |z| <r_1\},\\
								&M_{t,2}=\{z\in \overline{\mathbb C}\ |\ r_1<|z| <r_2\},\\
								&M_{t,3}=\{z\in \overline{\mathbb C}\ |\ |z|>r_2\}.
								\end{split}\]
								Here we denote by $r_1$ and $r_2$ the two positive solutions to \[(k-1)(e^{2t}r^{2k+2}+1)-(k+1)(r^{2k}+e^{2t}r^2)=0\] with $0<r_1<r_2$.
								Then we obtain two complete minimal disks $M_{t,1}$, $M_{t,3}$ and one complete minimal annulus $M_{t,2}$ in $\hh^4$.
								\item $[Y_t]|_{M_{t,1}}$ and $[Y_t]|_{M_{t,3}}$ are conformally equivalent to complete immersed, isotropic minimal disks $y_{t,1}$ and $y_{t,3}$ in $\hh^4$.
								Moreover,  $y_{t,1}$ and $y_{t,3}$ are isometrically congruent if and only if $t=0$.  $[Y_t]|_{M_{t,2}}$ is conformally equivalent to an immersed, complete, isotropic minimal annulus $y_{t,2}$ in $\hh^4$.
								\item
								For every fixed $k$, when $t\rightarrow +\infty$, $[Y_t]$ tends to
								a branched  $k-$cover of a totally geodesic surface $y_{\infty}$ $\SSS^4$ which is orthogonal to the equator $\SSS^3_0=\{x\in \SSS^4|x\perp (1,0,0,0,0)^t\}$.
								\item When $t\rightarrow +\infty$, $W([Y_t]|_{M_{t,1}})\rightarrow 0$.
								\item Set ${t_0}=\frac{1-k}{2}\ln k$. Then when $k$ is large enough,
								\begin{equation}\label{eq-estimate}
								W([Y_{t_0}]|_{M_{t_0,1}})\geq \frac{(k-1)\pi}{3}.
								\end{equation} Moreover, when $k\rightarrow+\infty$, $W([Y_{t_0}]|_{M_{t_0,1}})\rightarrow +\infty$. In particular for every $W_0\in\R^+$, there exists some   $k\in\mathbb Z^+$  with $k>2+\frac{3W_0}{\pi}$, and $t'\in\R$, such that $W([Y_t']|_{M_{t',1}})=W_0$.
							\end{enumerate}   \end{proposition}
							\begin{proof}
								(1). By Proposition \ref{prop-s4k}, we have
								$W([Y_t])=W([Y])=4\pi k$.

								The proof of (2)-(4) is the same as Proposition \ref{prop-h42}.  So let's focus on  (5) and (6).
								First we note that the Willmore energy of $[Y_t]|_{M_{t,j}}$ are (Here $b=\frac{k+1}{k-1}$)
								\begin{equation}\label{eq-yt-W-2}
								W(M_{t,j})=4\pi\int_{r_{j-1}}^{r_j}\
								\frac{k^2(k-1)^2 e^{2t} r^{2k-3}\left(e^{2t}r^{2k+2}+1 -b(r^{2k}+e^{2t}r^2)\right)^2}{\left(e^{2t}(1+r^{2k})^2 + k^2r^{2k-2}(1-e^{2t}r^2)^2 \right)^2}
								\dd r,~ j=1,2,3,
								\end{equation}
								with $r_0=0$, $r_3=+\infty$ and $r_1$ and $r_2$ as shown in the proposition.
								
								It is direct to check that
								\[\lim_{t\rightarrow+\infty}r_1=0 \hbox{ and } e^{2t}r_1^2\leq 1.\]
								When $t\rightarrow+\infty$ we have for $0\leq r\leq r_1$
								\[e^{2t}(1+r^{2k})^2 + k^2r^{2k-2}(1-e^{2t}r^2)^2 \geq e^{2t},~ e^{2t}r^{2k+2}+1 -b(r^{2k}+e^{2t}r^2)<1.\]
								So when $t\rightarrow+\infty$,
								\[\begin{split}
								\int^{r_{1}}_{0}\
								\frac{k^2 e^{2t} r^{2k-3}\left(e^{2t}r^{2k+2}+1 -b(r^{2k}+e^{2t}r^2)\right)^2}{\left(e^{2t}(1+r^{2k})^2 + k^2r^{2k-2}(1-e^{2t}r^2)^2 \right)^2}
								\dd r&\leq \int_{0}^{r_1}\frac{ 2k^2e^{2t}}{(e^{2t})^2}r^{2k-3} \dd r\\
								&=\frac{2k^2 e^{-2t}r_1^{2k-2}}{2k-2}\rightarrow0.\\
								\end{split}\]
								So for every fixed $k$, $\lim_{t\rightarrow+\infty}W([Y_t]|_{M_{t,1}})=0.$
								
								The key point of (6) is the technical estimate \eqref{eq-estimate}. We will leave the proof of it for the appendix.
							\end{proof}

							\subsection{Non-oriented examples of minimal Moebius strips in $\hh^4$}
							
							In this subsection, we consider some non-oriented minimal surfaces in $\mathbb{H}^4$, which is based on the work of \cite{DoWaSym2} and \cite{Wang-S4}.
							
							Set
							\begin{equation}
							h_1=\frac{3}{2} z^{5},\ h_2=-h_3= \frac{\sqrt{5}}{2} z^{3},\ \
							h_4=\frac{3}{2} z.\
							\end{equation}
							We have
							\begin{equation}\label{eq-y-mini-h}
							Y=\left(\begin{array}{c}
							y_0 \\
							y_1\\
							y_2 \\
							y_3 \\
							y_4\\
							y_5 \\
							\end{array}\right)
							=\left(\begin{array}{ccc}
							(r^{10}+5r^6+5r^4+1)  \\
							-(r^{10}-5r^6-5r^4+1) \\
							\sqrt{5}i(1+r^{6})(z^2-\bar{z}^2)\\
							\sqrt{5}(1+r^{6})(z^2+\bar{z}^2)\\
							-\sqrt{5}i(1-r^{4})(z^3-\bar{z}^3) \\
							\sqrt{5}(1-r^{4})(z^3+\bar{z}^3) \\
							\end{array}\right).
							\end{equation}
							with
							\[|\dd Y|^2=40r^2(4r^4-7r^2+4)(r^2+1)^4|\dd z|^2.\]
							So $Y$ has exactly two branched points $0$ and $\infty$. Consider $\mu(z)=-\frac{1}{\bar z}$, we have
							\[[Y(\mu(z))]=[Y(z)].\]
							As a consequence, $[Y]$ induces a branched Willmore $\R P^2$: $[Y]:S^2/{\mu}=\R P^2\rightarrow \SSS^{4}$ is a  Willmore $\R P^2$ with Willmore energy $12\pi$ and  one branched point at $z=0$. For more discussions on singularities and branched points of Willmore surfaces, see \cite{BR,Mi-R,KS2,KS1,LN,R}.

							Set $r_1=\frac{\sqrt{5}-1}{2}$ and \[M_1=\left\{z\in\C|0\leq r<r_1\right\}, \ M_2=\left\{z\in\C| r_1< r<\frac{1}{r_1}\right\}, \ M_3=\left\{z\in\bar\C|r>\frac{1}{r_1}\right\}.\]
							Set $\tilde{y}=\frac{1}{y_1}(y_0,y_2,y_3,y_4,y_5)^t.$
							We see that
							\begin{enumerate}
								\item $\tilde y|_{M_2/\mu}$ is a complete minimal Moebius strip in $\mathbb H^4$ with $W(y)=\frac{6\sqrt{5}\pi}{5}\approx10.733\pi$.
								\item   $\tilde y|_{M_1}=(\tilde y\circ\mu )|_{M_3}$ is a branched minimal disk in $\mathbb H^4$ with Willmore energy $W(y)=12\pi(1 - 2\sqrt 5/5)\approx 1.267\pi$ and one branched point $z=0$.
							\end{enumerate}
							
							It is natural to ask whether the complete minimal Moebius strip $\tilde y|_{M_2}$ takes uniquely the minimum of the Willmore energy among all complete minimal Moebius strips in $\hh^n$, $n\geq 4$.
							
							\section{Remarks on the non-rigidity of isotropic surfaces in $\SSS^4$}
							
							Finally we would like to discuss briefly some simple applications of the W-deformations
							on the study of stability problems of Willmore surfaces and minimal surfaces. More detailed study will be done in a separate publication, since it will involve many other independent calculations.
							We refer to \cite{NS2, Simons, Urbano, Weiner} for more details on this topics, in particular Theorem 3.3.1 and Corollary 3.3.1 of \cite{Simons}.
							
							Since for isotropic surfaces in $\SSS^4$, we have an explicit W-representation formula, we see that  W-deformations are globally defined if the surfaces are globally defined. From this we see immediately that they are Willmore non-rigidity since they admits  non-trivial Willmore Deformations.
							
							\begin{theorem}Let $y:M\rightarrow \SSS^4$ be an isotropic (hence Willmore) surface from a closed Riemann surface $M$  with its conformal Gauss map in $\mathcal{M}_L$. Then $y$ is  Willmore non-rigid. That is, it admits conformal Jacobi fields different from the conformal  Killing fields  which come from conformal transformations of $S^4$.
							\end{theorem}
							\begin{proof}
								We first consider the case that $y$ is not conformally equivalent to a minimal surface in $\SSS^4$.
								By Theorem 3.7,  the condition that the conformal Gauss map of $y$ is in $\mathcal{M}_L$, is equivalent to saying  that it is coming from a $K^{\C}-$dressing of some minimal surface in $\SSS^4$.
								Therefore by Theorem \ref{th-willmore-iso-formula}, there exists a  family of Willore surfaces $y_t$ such that $y_t$ is real analytic in $t$ and $y_{t}|_{t=0}=y$ and $y_t|_{t=t_0}$ is a minimal surface in $\SSS^4$. So $\{y_t\}$ does not come from any conformal transformations of $\SSS^4$ and the Jacobi field of $y_t$ is not a conformal Killing field.
								
								Now consider the case that $y$ is conformally equivalent to a minimal surface in $\SSS^4$.
								Without lose of generality, we assume $y$ has the potential as the form in Proposition \ref{prop-s41}.	By Proposition \ref{prop-s41} and Theorem \ref{th-willmore-iso-formula}, there exists globally a family of Willore surfaces $y_t$ such that $y_t$ is real analytic in $t$ and $y_{t}$ is not conformally equivalent to any minimal surface in $\SSS^4$ when $0<t<\pi/2$.  So $\{y_t\}$ does not come from any conformal transformations of $\SSS^4$ and the Jacobi field of $y_t$ is not a conformal Killing field.
								
							\end{proof}
							
							For minimal surfaces in $\SSS^4$, we also have the following
							\begin{theorem}Let $y:M\rightarrow \SSS^4$ be an isotropic minimal surface from a closed Riemann surface $M$. Then $y$ is non-rigidity. That is, it admits Jacobi fields different from the  Killing  fields which come from isometric transformations of $S^4$.
							\end{theorem}
							
							\begin{proof}
								Assume without loss of generality the normalized potential of $y$ is of the form \eqref{eq-b1} with $h_2=h_3$.  	
								Let \[\hat T_t=\left(
								\begin{array}{cccc}
								I_4& 0 & 0\\
								0& \cosh t& i\sinh t \\
								0& -i\sinh t & \cosh t  \\
								\end{array}
								\right)\]
								be a one-parameter subgroup of $K^{\C}$.
								The one-parameter family of normalized potentials $\eta_t$ has the same form as $y$ in \eqref{eq-b1}, except the functions $\{h_j\}$ becomes  $\{e^th_j\}$. Substituting  $\{e^th_j\}$ into \eqref{eq-willmore in s4-y-1}, we obtain the Willmore family $y_t=\frac{1}{y_{0t}}(y_{1t},,y_{2t},y_{3t},y_{4t},y_{5t})$ derived by $\eta_{t}$.  We have that $y_t$ is real analytic in $t$ and for every $t$, $y_t$ is a minimal surface in $S^4$.
								
								Let $t$ tends to $+\infty$. We have that $y_t$ tends to a conformal map into $S^2$. As a consequence,
								$y_t$ can not be derived by an isometric transformations of $S^4$.  Hence the  Jacobi field of $y_t$ is not a Killing field of $y$.
							\end{proof}
							
							We refer to Ejiri's interesting paper  \cite{Ejiri1983} for the discussion of the index of minimal two-spheres in $S^{2m}$. Note that the Willmore deformations contribute explicitly to the index of minimal two-spheres in $S^{4}$ \cite{NS2,Weiner}.
							
							\section{Appendix: Proof of \eqref{eq-estimate}}

							Set $a=e^{2t_0}= k^{-(k-1)}$, $\rho=r^2$. Set $L=a\rho^{k+1}+1-b\rho^{k}-ab\rho$ with $b=\frac{k+1}{k-1}$. Let $\rho_1\in(0,1)$ and $\rho_2\in(1,+\infty)$ be the two solutions to \[L(\rho)=a\rho^{k+1}+1-b\rho^{k}-ab\rho=0.\]
							We can rewrite $W(M_{t_0,1})$ as
							\[W(M_{t_0,1})=2\pi\int^{\rho_1}_{0}  \frac{ak^2(k-1)^2\rho^{k-2}\left(a\rho^{k+1}+1-b\rho^{k}-ab\rho)\right)^2}{\left(a(1+\rho^{k})^2 + k^2\rho^{k-1}(1-a\rho)^2 \right)^2}
							\dd \rho.\]
							Then \eqref{eq-estimate} follows from the following Lemma.
							\begin{lemma}
								\begin{enumerate}
									\item When $k\rightarrow+\infty$, $\rho_1>e^{-3/k^2}$; In particular
									\[\lim_{k\rightarrow\infty}\rho_1=\lim_{k\rightarrow\infty}(\rho_1)^k=1.\]
									\item On $[0,\rho_1]$, $L(\rho)\geq \rho_1^{-1}(\rho_1-\rho).$ When $k\rightarrow+\infty$,
									\begin{equation}\label{eq-estimateofW}
									W(M_{t_0,1})\geq \frac{2\pi k^2(k-1)^2 }{\rho_1^2}I_1, \hbox{ with } I_1= \int^{\rho_1}_{0}\frac{ a\rho^{k-1}(\rho_1-\rho)^2}{\left(2a + k^2\rho^{k-1}  \right)^2}
									\dd \rho.
									\end{equation}
									\item
									Set $\varphi=\rho/\rho_1$. Then $I_1$ is tending to
									\[I_2=\int^{1}_{0}\frac{ a\varphi^{k-1}(1-\varphi)^2}{\left(2a + k^2\varphi^{k-1}  \right)^2} \dd \varphi\]
									when $k\rightarrow+\infty.$
									\item  $I_2>\frac{1}{9}R(a,k)$ with $\delta=(\frac{a}{k^2})^{\frac{1}{k-1}}\in(0,1)$ and
									\[R(a,k)=\frac{1}{k^2}\left(\frac{2}{k-1}-\frac{2\delta}{k}+\frac{\delta^2}{k+1}\right)+ \frac{1}{k^2}\left(-\frac{2\delta}{k-2}+\frac{\delta^2}{k-3}\right)-\frac{a}{k^4}\left(\frac{1}{k-1}-\frac{2}{k-2}+\frac{1}{k-3}\right).\]
									Moreover, when $k\rightarrow+\infty$,
									$R(a,k)=\frac{2}{k^2(k-1)}+o\left(\frac{1}{k^3}\right).$
									\item
									When $k$ is large enough,
									$W(M_{t_0,1})> \frac{1}{3}(k-1)\pi.$
								\end{enumerate}
							\end{lemma}
							\begin{proof}
								(1). From $0<\rho_1<1$ and $a\rho_1^{k+1}+1-b\rho_1^{k}-ab\rho_1=0$, we have
								\[(\rho_1)^{k+1}=\frac{1-ab\rho_1}{b\rho_1^{-1}-a}\geq\frac{1-ab}{b-a}=1+\frac{(1-b)(1+a)}{b-a}=1+\frac{1+a}{b-a}\frac{2}{k+1}.\]
								From this, $\lim_{k\rightarrow\infty}\rho_1=\lim_{k\rightarrow\infty}(\rho_1)^k=1.$
								
								(2). Since $L'(\rho)=a(k+1)\rho^{k}-bk\rho^{k-1}-ab,\ \ L''(\rho)=ak(k+1)\rho^{k-1}-bk(k-1)\rho^{k-2}=ak(k+1)\rho^{k-2}(\rho-1).$ So on $(0,\rho_1)$, $L''(\rho)<0$, from which we have $L(\rho)\geq \rho_1^{-1}(\rho_1-\rho)$. And \eqref{eq-estimateofW} follows from this and the fact that $a(1+\rho^{k})^2 + k^2\rho^{k-1}(1-a\rho)^2<2a+k^2\rho^{k-1}$.
								
								(3). Since $\rho=\rho_1\varphi$, we have
								\[I_1=\rho_1^{k+1}\int^{1}_{0}\frac{ a\varphi^{k-1}(1-\varphi)^2}{\left(2a + k^2\rho_1^{k-1}\varphi^{k-1}  \right)^2} \dd \varphi.\]
								Since $\lim_{k\rightarrow+\infty}\rho_1=\lim_{k\rightarrow+\infty}\rho_1^{k+1}=\lim_{k\rightarrow+\infty}\rho_1^{k-1}=1$, we have
								\[1<\frac{2a+k^2\varphi^{k-1} }{2a + k^2\rho_1^{k-1}\varphi^{k-1}}< \frac{1}{\rho_1^{k-1}}\rightarrow 1\]
								as $k\rightarrow+\infty$. (3) follows from this.
								
								(4) First we have $k^2\delta^{k-1}=a$. So
								\[2a+k^2\varphi^{k-1}<3a,~~ \forall\varphi\in(0,\delta);\ ~ 2a+k^2\varphi^{k-1}<3k^2\varphi^{k-1},~~ \forall\varphi\in(\delta,1).\]
								By substituting $\delta^{k-1}=\frac{a}{k^2}$, we have
								\[\begin{split}
								I_2&>\int^{\delta}_{0}\frac{ a\varphi^{k-1}(1-\varphi)^2}{9a^2} \dd \varphi+\int^{1}_{\delta}\frac{ a\varphi^{k-1}(1-\varphi)^2}{9k^4\varphi^{2(k-1)}} \dd \varphi\\
								&= \frac{\delta^{k-1}}{9a}\left(\frac{1}{k-1}-\frac{2\delta}{k}+\frac{\delta^2}{k+1}\right)+ \frac{a\delta^{1-k}}{9k^4}\left(\frac{1}{k-1}-\frac{2\delta}{k-2}+\frac{\delta^2}{k-3}\right)\\
								&~~-\frac{a}{9k^4}\left(\frac{1}{k-1}-\frac{2}{k-2}+\frac{1}{k-3}\right)\\
								&= \frac{1}{9}R(a,k).\\
								\end{split}
								\]
								When $k\rightarrow+\infty$, $\delta= k^{-1-\frac{2}{k-1}}\rightarrow0$ and hence
								$R(a,k)=\frac{2}{k^2(k-1)}+o\left(\frac{1}{k^3}\right).$
								This finishes (4).
								
								(5) As a consequence, we have
								\[W(M_{t_0,1})\geq \frac{2\pi k^2(k-1)^2 }{\rho_1^2}I_1\geq\frac{(4-\varepsilon)\pi}{9}(k-1),\]
								for some $\varepsilon\in(0,1/2)$ when $k\rightarrow+\infty$, which finishes (5).
							\end{proof}

							\def\refname{Reference}

						\end{document}